\documentclass[11pt]{article}
\usepackage[total={15cm,21cm}]{geometry}

\usepackage{amsmath,amssymb,amsfonts}
\usepackage[alphabetic,y2k,initials]{amsrefs}

\usepackage[usenames,dvipsnames]{xcolor}
\usepackage{graphics}
\usepackage{float}

\usepackage{tikz}
\usepackage{pgfplots,tikz-3dplot}

\usepackage{pdfsync}
\usepackage{hyperref}
\usepackage{multirow}

\usepackage{longtable}

\usepackage{calrsfs}
\DeclareMathAlphabet{\pazocal}{OMS}{zplm}{m}{n}
\newcommand{\LL}{\mathcal{L}}

\newcommand{\Q}{\pazocal{Q}}
\newcommand{\E}{\pazocal{E}}
\newcommand{\I}{\pazocal{I}}
\newcommand{\Pol}{\mathcal{P}}
\newcommand{\T}{\mathcal{T}}
\newcommand{\R}{\mathbf{R}}

\newcommand{\Curve}{\mathcal{C}}

\newcommand{\refeq}{\eqref}

\def\rank{\mathrm{rank}}
\def\sign{\mathrm{sign\,}}
\def\deg {\mathrm{deg\,}}

\newtheorem{theorem}{Theorem}[section]
\newtheorem{proposition}[theorem]{Proposition}
\newtheorem{corollary}[theorem]{Corollary}
\newtheorem{lemma}[theorem]{Lemma}
\newtheorem{remark}[theorem]{Remark}
\newtheorem{example}[theorem]{Example}
\newtheorem{definition}[theorem]{Definition}

\newenvironment{proof}[1][Proof]{\noindent\textit{#1.} }{\hfill$\Box$\newline\medskip}

\numberwithin{equation}{section}

\title{Resonance of ellipsoidal billiard trajectories and extremal rational functions}

\usepackage{authblk}
\author[1,3]{Vladimir Dragovi\'c}
\author[2,3]{Milena Radnovi\'c}
\affil[1]{\textsc{The University of Texas at Dallas, Department of Mathematical Sciences}}
\affil[2]{\textsc{The University of Sydney, School of Mathematics and Statistics}}
\affil[3]{\textsc{Mathematical Institute SANU, Belgrade}}
\affil[ ]{\texttt{vladimir.dragovic@utdallas.edu, milena.radnovic@sydney.edu.au}}

\date{}

\begin{document}

\maketitle

\begin{abstract}
We study resonant billiard trajectories  within quadrics in the $d$-dimensional Euclidean space. We relate them to the theory of approximation, in particular the extremal rational functions on the systems of $d$ intervals on the real line. This fruitful link enables  us to prove fundamental properties of the billiard dynamics and to provide a comprehensive study of a large class of non-periodic trajectories of integrable billiards. A key ingredient is a functional-polynomial relation of a generalized Pell type. Applying further these ideas and techniques to $s$-weak billiard trajectories, we come to a functional-polynomial relation of the same generalized Pell type.


\emph{AMS subclass:} 14H70, 41A10, 70H06, 37J35, 26C05
\end{abstract}
\newpage
\tableofcontents

\section{Introduction}\label{sec:intro}
	{\color{black}
	Mathematical billiard is a dynamical system where a particle moves along straight lines within certain bounded domain in the space, and obeys \emph{billiard reflection law} when it hits the boundary {\color{black} \cite{KTBilliards, Tab}}.
It is clear that the behaviour of such a system crucially depends on the shape of the domain. {\color{black} 
Billiards within quadrics are of a special interest because of their rich geometric and dynamical properties and, in particular, \emph{integrability}. While the two-dimensional systems of billiards within conics, and especially ellipses, have been intensively studied from various perspectives, see for example \cite{bol, KTBilliards, DragRadn2011book, DR1, ADSK, KS, bm, bm1, glu1} and references therein, the amount of research of their higher-dimensional generalizations is significantly smaller. 
In this paper, our main object are the billiard systems within quadrics in dimensions higher than two.} }	

The study of billiards within an ellipsoid in three-dimensional space started with Darboux in 1870,
see \cite{Dar1870}. After a long gap, the interest to the subject renewed starting from 1990's,
see \cites{MV1991, CCS1993, DragRadn1998a, DragRadn1998b,  Wiersig2000, Fed2001, BDFRR2002, WDullin2002,  DragRadn2004, AbendaFed2006, DragRadn2006jmpa, DragRadn2008, AbendaGrin2010, DragRadn2011book,CRR2011, DragRadn2012adv, RR2014}.

{\color{black}
One of the natural first questions when studying billiards is to find about their periodic trajectories.
For ellipsoidal billiards, according to the generalized Poncelet theorem \cites{Dar1870,CCS1993}, if there is a periodic billiard trajectory within ellipsoid, then each other trajectory sharing the same \emph{caustics} will be periodic with the same period.
The existence of caustics for ellipsoidal billiard is a geometric manifestation of the integrals of motion: any trajectory of billiard within an ellipsoid in the $d$-dimensional Euclidean space has $d-1$ caustics, which are quadrics confocal with the boundary.

Analytic conditions for periodicity are, in the planar case, the classical Cayley's conditions \cites{Cayley1854,GrifHar1978}, while for higher dimensions they were first derived in \cites{DragRadn1998a,DragRadn1998b}, see also \cite{DragRadn2004}.
The relationship of those conditions with polynomials was explored in \cite{RR2014}, and further deepened} in \cite{DragRadn2018}, where the connection between periodic trajectories of billiards within ellipsoid in $d$ dimensional Euclidean space and extremal polynomials on $d$ intervals on the real line
was fully exploited, {\color{black}enabling us to closely investigate the behaviour of the rotation numbers and led to} a complete classification of periodic trajectories, see also \cites{DragRadn2019rcd, ADR2019rcd, ADR2020rj}. {\color{black}Namely, the condition for $n$-periodicity of the billiard trajectories in the $d$-dimensional case} was related to the existence {\color{black} of} a pair of real polynomials
$\hat{p}_n$, $\hat{q}_{n-d}$
of degrees $n$ and $n-d$ respectively such that \emph{the polynomial Pell equation} holds:
\begin{equation}\label{eq:pell}
\hat{p}_n^2(s)-\hat{\mathcal{P}}_{2d}(s)\hat{q}_{n-d}^2(s)=1,
\end{equation}
where $\hat{\mathcal{P}}_{2d}(s)$ is a polynomial of degree $2d$ determined by the axes of the boundary ellipsoid and the caustics of the trajectory. Apparently, the polynomials  $\hat{p}_n$ are Chebyshev polynomials, the extremal polynomials on the system of $d$ intervals of the real line, where the endpoints of the intervals are the zeros of the polynomial $\hat{\mathcal{P}}_{2d}(s)$.

Thus, one can conclude that those existing research results provided a well-rounded understanding of periodic ellipsoidal
billiard trajectories.
The goal of this paper is to {\color{black}deepen the insight into} \emph{nonperiodic} trajectories.
{\color{black} As noted above,} the billiards in higher dimensions are much less studied than their planar counterparts.
{\color{black}One of the reasons contributing to increased complexity {\color{black} of billiards in dimensions three and higher} is the existence of skew lines, making} the study
of higher-dimensional billiards both challenging and interesting.

{\color{black}
Start with the following example of the classification of billiard trajectories.
Noticing that any pair of lines in the space can be a pair of coinciding lines, coplanar lines, or skew lines, we note that any billiard trajectory belongs exactly to one of the following classes:
\begin{itemize}
	\item periodic trajectories; or
	\item nonperiodic trajectories that have a pair of nonconsecutive coplanar segments;
	\item trajectories where all pairs of nonconsecutive segments are skew.
\end{itemize}
For {\color{black} the dimension $d\ge3$, we refined the last class in \cites{DragRadn2008} using the following property (see \cite{Tjurin1975} and also \cite{DragRadn2011book}):}
For any pair of lines $\ell$, $\ell'$ in the $d$-dimensional space, which are touching the same set of $d-1$ confocal quadrics, we can assign a set of at most $d-2$ lines: $\ell_1$, \dots, $\ell_j$, $j\le d-2$, such that:
\begin{itemize}
 \item $\ell_1$, \dots, $\ell_j$ are touching the same $d-1$ quadrics as $\ell$, $\ell'$;
 \item in the sequence $\ell$, $\ell_1$, \dots, $\ell_j$, $\ell'$, each two consecutive lines are coplanar.
\end{itemize}
Based on the number $j$ of such lines  $\ell_1$, \dots, $\ell_j$, in \cite{DragRadn2008}, we introduced classes of skew lines, invented the concept of \emph{weak periodic} ellipsoidal billiard trajectories, and derived analytic conditions for such trajectories.
}

Our {\color{black}aim} in this paper is to
{\color{black}initiate investigation of} \emph{resonant} billiard trajectories  within ellipsoids in $d$-dimensional Euclidean space  and to relate them, {\color{black}on one side}, to the theory of approximation, in particular the extremal \emph{rational} functions on the systems of $d$ intervals on the real line, {\color{black}and on the other side, to weak perodicity which originated in \cite{DragRadn2008}}.

 A key ingredient in this work is a functional-polynomial relation of a \emph{generalized} Pell type:
\begin{equation}\label{eq:genPell}
{\color{black}
A^2(z)-\prod_{j=1}^{2d}(z-c_j)B^2(z)=S_m(z)^2.}
\end{equation}
Note that, in this generalized equation, there is the square of a polynomial on the right-hand side, while the original polynomial Pell equation \ref{eq:pell} has just a constant instead.

The link we provide enables us to derive fundamental properties of the billiard dynamics and to provide a comprehensive study of a large class of non-periodic trajectories of such billiard systems.


\subsection{The organization of the paper}

This paper is organized as follows.

The following Section \ref{sec:Em} starts with recalling the notion of $(E, m)$-representation from
{\color{black}\cite[Section 5.5.4]{KLN1990}}
and a brief review of the relationship with the extremal polynomials on $d$ real intervals. Then, in Section \ref{sec:EP} we formulate the main new extremal problem, so-called {\it restricted extremal problem} for rational functions of bounded degree of the denominator on a system of $d$ real intervals. The main result of that section is Theorem \ref{th:pell-unique} which provides the rigidity and uniqueness conditions for a system of $d$ intervals to admit solutions of the restricted extremal problem with a prescribed degree of the denominator $g_m$, i.e.~which allows a solution of a generalized Pell equation of a given form, see equation \eqref{eq:genPell}, with the degree of the polynomial $S_m$ being equal to $g_m$. In Lemma \ref{lemma:alternance} we derive the cardinality of alternance sets for the extremal rational functions.

In Section \ref{sec:rresonant} we introduce the resonance of an ellipsoidal billiard trajectory and associated winding numbers, see Definition \ref{def:r-resonant}. {\color{black} Theorem \ref{th:r-resonant} provides a generalized Pell's equation \eqref{eq:r-resonant} which relates the degree
of the polynomial on the  left hand side with the resonance of trajectories with given caustics.}
Definition \ref{def:adjoint} introduces the adjoint resonance, the adjoint winding numbers {\color{black} and the adjoint type of a path of a billiard trajectory}. Proposition \ref{th:winding}
proves that the adjoint winding numbers are strictly decreasing, while Theorem \ref{th:length-periodic} proves that all periodic trajectories with given caustics have the same Euclidean length and provides the formula for the length. 
In Section \ref{sec:uniquenesscaustics}, we prove the uniqueness result for the caustics of the trajectories with a given adjoint resonance, see Theorem \ref{th:signature-caustics}. The proof is based on the so-called Audin Alternative as formulated in Lemma \ref{lemma:same-type-caustics}  and  Theorem \ref{th:pell-unique} from Section \ref{sec:EPEm}.

Section \ref{sec:sweak} is devoted to the weak periodicity, {\color{black}a} natural stratification of non-periodic
ellipsoidal billiard trajectories from \cite{DragRadn2008}, see Definition \ref{def:s-skew-d} in Section \ref{sec:sweakdPell}. The main result is Theorem \ref{the:pell-weak-d} which characterizes weak periodic trajectories in terms of generalized Pell equations \eqref{eq:pell-weak-d}. Both resonant trajectories from Section \ref{sec:rresonant} and weak periodic trajectories from Section \ref{sec:sweak} are characterized by the generalized Pell equations of the same form, see equation (\ref{eq:genPell}). A slight difference is that there are more restrictive conditions on locations of zeros of the involved polynomials in Theorem \ref{th:r-resonant} than in Theorem \ref{the:pell-weak-d}. As a result, we get Theorem \ref{the:resonant-weak-d} with the inequality \eqref{eq:resonanat-skew}.
The last two {\color{black}subsections} are devoted to examples of weak periodic trajectories in dimensions $3$ and $4$.

\section{The extremal problem and $(E, m)$-representation}\label{sec:EPEm}
We first review the notion of $(E, m)$-representation from {\color{black}\cite[Section 5.5.4]{KLN1990}} and then formulate the main extremal problem considered in this paper.

\subsection{$(E, m)$-representation}\label{sec:Em}

Let $c_{2d}<c_{2d-1} < \dots <c_1$ be real numbers and $\pazocal{P}(z)$ a real polynomial of degree $2n$ which is positive on the set
$E
=\cup_{j=1}^{d}[c_{2j}, c_{2j-1}].
$
We say that such a polynomial admits an \emph{$(E, m)$-representation}, {\color{black} $m\ge d$}, if there are  real polynomials $A_m(z)$ and {\color{black}$B_{m-d}(z)$} of degrees $m$ and $m-d$ respectively such that:
\begin{itemize}
	\item
	$
\pazocal{P}(z)
=
A_m^2(z)-\prod_{j=1}^{2d}(z-c_j)B_{m-d}^2(z);
$
\item
all zeros of $A_m(z)$ and {\color{black}$B_{m-d}(z)$} are {\color{black} real,} simple, and belong to $E$;
\item in each of the $d$ {\color{black}open} intervals {\color{black} $(c_{2j}, c_{2j-1})$, $j=1, 2, \dots, d,$} polynomials $A_m(z)$ and $B_{m-d}(z)$ satisfy the following:
\begin{itemize}
	\item the polynomial $A_m(z)$ has one zero more than {\color{black}$B_{m-d}(z)$}; and
	\item the zeros of those two polynomial alternate: between any two zeros of $A_m(z)$ there is exactly one zero of $B_{m-d}(z)$.
\end{itemize}
\end{itemize}

{\color{black}
\begin{remark} These requirements imply that $A_{m}(c_j)\ne 0$, $B_{m-d}(c_j)\ne 0$ for all $j=1,\dots, 2d$.
\end{remark}
}

We quote one relatively recent characterization of polynomials which admit $(E, m)$-representation.

\begin{theorem}[Krein, Levin, Nudelman,
	{\color{black}\cite[Th. 5.13]{KLN1990}}]\label{th:kln}
Let $\pazocal{P}(z)$ be a polynomial of degree $2n$ which is positive on $E$,
and the numbers $N_k$ determined by the following $d-1$ relations:
$$
\frac{1}{2\pi}\int_E\frac{z^j \ln \pazocal{P}(z)}{\sqrt{T(z)}}dz
=\sum_{k=1}^{d-1}
(-1)^kN_k
\int_{c_{2(d-k)+1}}^{c_{2(d-k)}}\frac{z^j}{\sqrt{|T(z)|}}dz
+ (-1)^d m\int_{c_1}^\infty \frac{z^j}{\sqrt{|T(z)|}}dz,
$$
where $j\in\{0, 1,\dots, d-2\}$ and $T(z)= \prod_{j=1}^{2d}(z-c_j)$.

Then $\pazocal{P}(z)$ admits an $(E, m)$-representation if
and only if the numbers $N_k$ are positive integers.
Moreover, then each $N_k$ represents the number of zeros of $A_m(z)$ in the open interval $(0, c_{2(d-k+1)-1})$.
\end{theorem}

Now, we are going to review some facts from the theory of extremal polynomials on a system of several intervals of real axis. Such polynomials are also called generalized  Chebyshev polynomials, since in the case of one interval, they coincide with the classical Chebyshev polynomials.

For $n=0$, the polynomial $\pazocal{P}(z)$ is constant, so we can assume $\pazocal{P}(z) \equiv 1$, and get the Pell equation:

\begin{equation}
1=A_m^2(z)- T(z)B_{m-d}^2(z).
\end{equation}

The polynomials $A_m$ are, in this case, up to a scalar factor, the extremal polynomials on the system of $d$ intervals $E$ with respect to the uniform norm. They were introduced and studied initially by Chebyshev and his school (see \cite{AhiezerAPPROX}). The Pell equation indicates that the system $E$ is the maximal one with $A_m$ as its extremal polynomial of degree $m$. We are going to study the structure of extremal points of $A_m$, in particular the set of points of alternance.

{\color{black} Still assuming $n=0$ and $\pazocal{P}(z) \equiv 1$, we} notice that the roots of $T$ are simple solutions of the equation $A_m^2(z)=1$, while the roots of
$B_{m-d}$ are double solutions of the equation $A_m^2(z)=1$.
{\color{black} Due to the assumptions on} the degrees of the polynomials, these are all points where $|A_m(z)|$ is equal to $1$.

A set of \emph{points of alternance} is, by definition, a subset of the solutions of the equation $|A_m(z)|=1$, with the maximal number of elements, such that the signs of $A_m$ alter on it.
Such a set is not uniquely determined, however the number of its elements {\color{black} is} fixed and equal to $ m+1$, \cite{SodinYu1992, Si2015}.

\subsection{The Extremal Problem}\label{sec:EP}

One of the extremal problems considered in {\color{black}\cite[Section 5.5.5]{KLN1990}} is as follows:
\begin{quotation}
Let $\pazocal P(z)$ be a real polynomial of degree $2n$ which is positive on $E$ and $m$ a given natural number, $m\ge n$. Consider the rational functions of the form
\begin{equation}
\frac {X(z)}{H(z)\pazocal P(z)},
\end{equation}
where $X(z)$ and $H(z)$ are real monic polynomials of degrees {\color{black} at most} $m+d-1$ and $d-1$ respectively. 
We assume that all zeros of $H(x)$ belong to the {\color{black}open} gap-intervals $(c_{2j-1}, c_{2j-2})$ and each of those intervals contains at most one zero of $H(z)$.
Such rational functions are called {\it $mPE$ rationals}. Find among $mPE$ rationals the one {\color{black} with the least deviation} from zero on $E$ with respect to the uniform
norm.	
\end{quotation}
 
The solution  of the Extremal Problem from {\color{black}\cite[Section 5.5.5]{KLN1990}} establishes first that a generalized Pell's equation

\begin{equation}
{\color{black}
S_m(z)^2\pazocal P(z)=A_{m+g_m}^2(z)-\prod_{j=1}^{2d}(z-c_j)B_{m+g_m-d}^2(z),}
\end{equation}
has always a solution without any restrictions on the set $E$. Here, $S_m(z)$ is a real {\color{black} monic} polynomial of degree $g_m\le d-1$ with all zeros in the {\color{black}open} gap-intervals $(c_{2j-1}, c_{2j-2})$ with at most one zero in each gap-interval, and which generates an $(E, m+g_m)$
representation {\color{black} of the polynomial $S_m(z)^2\pazocal P(z)$}. {\color{black} The polynomials $A_{m+g_m}(z)$ and $B_{m+g_m-d}(z)$ are of the degrees $m+g_m$ and $m+g_m-d$ respectively.}
Then they conclude that the extremal $mPE$ rational is
$$
\chi_0(z)=\frac{{\color{black} A_{m+g_m}(z)}}{A_0S_m(z)},
$$
where $A_0$ is the leading coefficient of ${\color{black}A_{m+g_m}(z)}$.

We are going to consider the following restricted version of the above extremal problem:

{\bf The Restricted Extremal Problem.}  Let $\pazocal P(z)$ be a polynomial of degree $2n$ which is positive on $E$, and $m$, $q$ integers such that
$m\ge n$, {\color{black} $m\ge d$},  and $0\le q\le d-1$.
Under which conditions on $E$ the denominator of the extremal  $mPE$ rational has the degree $g_m$ which is not greater than $q$?  We will call such extremal  $mPE$ rationals {\it $q$-restricted} if the degree of their denominators $g_m$ do not exceed $q$.

{\color{black} In the further discussion, we will assume that $n=1$ and $\pazocal P\equiv 1$.}

Suppose two {\color{black}unions of closed} intervals are given:
$$
\I_d=\bigcup_{j=0}^{d-1}[c_{2(d-j)}, c_{2(d-j)-1}]
\quad\text{and}\quad
\I_d^*=\bigcup_{j=0}^{d-1}[c_{2(d-j)}^*, c_{2(d-j)-1}^*]
$$
which both admit extremal $g_m$-restricted $m$-rationals which are generated through the solutions of the generalized Pell equations:
\begin{equation}\label{eq:genPellA}
{\color{black}
S_m(z)^2=A_{m+g_m}^2(z)-\prod_{j=1}^{2d}(z-c_j)B_{m+g_m-d}^2(z),}
\end{equation}
and
{\color{black}
$$
\tilde S_m(z)^2=\tilde A_{m+g_m}^2(z)-\prod_{j=1}^{2d}(z-c_j^*)\tilde B_{m+g_m-d}^2(z),
$$
where  $\deg S_m=\deg \tilde S_m=g_m\le q$, {\color{black}while} $(A_{m+g_m}, B_{m+g_m-d})$ and $(\tilde A_{m+g_m}, \tilde B_{m+g_m-d})$
form $(\I_d, m+g_m)$ and $(\I_d^*, m+g_m)$  representations} respectively.

\begin{theorem}\label{th:pell-unique}
{\color{black}
Let $g_m$ be an integer, $0\le g_m\le d-1$. Suppose that:
\begin{itemize} \item [\textrm{(i)}] at least $g_m+1$ of the intervals from {\color{black}the union} $\I_d$ coincide with corresponding intervals from $\I_d^*$;
\item[\textrm{(ii)}] for each $j\in \{0, \dots, d-1\}$, $c_{2(d-j)}= c_{2(d-j)}^*$ or $c_{2(d-j)-1}= c_{2(d-j)-1}^*$;
\item[\textrm{(iii)}] in each pair of the corresponding {\color{black}closed intervals}
$$
[c_{2(d-j)},c_{2(d-j)-1}]
\quad\text{and}\quad
[c_{2(d-j)}^*,c_{2(d-j)-1}^*]
$$
the rational functions  $A_{m+g_m}/S_m$, $\tilde A_{m+g_m}/\tilde S_m$  have the same number of extreme points;
\item[\textrm{(iv)}] the functions $A_{m+g_m}/S_m$, $\tilde A_{m+g_m}/\tilde S_m$ have the same uniform norms on $\I_d$ and $\I_d^*$ respectively;
{\color{blue}
\item[\textrm{(v)}] in each pair of the corresponding {\color{black}open} gap intervals
$$
(c_{2(d-j)-1},c_{2(d-j)-2})
\quad\text{and}\quad
(c_{2(d-j)-1}^*,c_{2(d-j)-2}^*)
$$
}
{\color{black}either both polynomials $S_m$ and $\tilde S_m$ have a zero or none of them has a zero.
}
\end{itemize}
Then $\I_d=\I_d^*$ and $A_{m+g_m}/S_m=\pm\tilde A_{m+g_m}/\tilde S_m$.
}
\end{theorem}
\begin{proof}
{\color{black}
Denote $\mathcal {A}_{m}=A_{m+g_m}/S_m$, $\tilde {\mathcal A_m}=\tilde A_{m+g_m}/\tilde S_m$, $I_{j}=[c_{2(d-j)},c_{2(d-j)-1}]$, and $I_{j}^*=[c_{2(d-j)}^*,c_{2(d-j)-1}^*]$.

Let $L_{j}$ be the number of points of alternance
of  $\mathcal A_m$ on $I_j$.
Whenever $I_j=I_j^*$, the assumption (iii) implies that  $\tilde {\mathcal A_m}$ will also have $L_{\alpha_j}$ points of alternance on $I^*_{j}$.

Denote by $Z(I)$ the number of zeros of zeros of the difference $\mathcal A_m - \tilde {\mathcal A_m}$ on a closed interval $I$.
We will estimate $Z(I)$ for various closed intervals $I$.
One can easily see that:
\begin{gather*}
Z(I_j)\ge L_j \quad\text{whenever}\quad I_j=I_j^*,
\\ 
Z(I_j)\ge L_j-1 \quad\text{whenever}\quad I_j\neq I_j^*.
\end{gather*}
The proofs of those two inequalities follow along the lines of the proof of} Lemma 2.11 from \cite{PS1999}. 
Namely, for the first inequality one needs
 to select a proper sign of $\tilde {\mathcal A_m}$ on the first coinciding interval. 
 From the assumption (v) then it follows that the signs of $\tilde {\mathcal A_m}$ at all other coinciding intervals are such to guarantee the fulfilment of the first inequality.

Observe that
$$
\sum_{j=1}^d L_j=m+g_m+d,
$$
 since the extremal points are the $m+g_m-d$ zeros of the polynomial $B_{m+g_m-d}$ and the $2d$ endpoints of the segments.
On the other hand, from the previous inequalities we get
$$
Z([c_{2d}, c_1])\ge \sum_{j=1}^d L_j-(d-g_m-1)=m+2g_m+1.
$$
Since the degree of the numerator of the difference $\mathcal A_m - \tilde {\mathcal {A}_m}$ is at most $m+2g_m$
we conclude that $\mathcal A_m =\tilde {\mathcal A_m}$ and $\pazocal{I}_d=\pazocal{I}_d^*$.
\end{proof}

\begin{remark} The case $g_m=d-1$ was considered in
	{\color{black}\cite[Section 5.5.5]{KLN1990}}. Another extreme case is $g_m=0$. It  was considered in
{\color{black}\cite[Theorem 2.12]{PS1999}}.
\end{remark}
In the following two propositions, we are going to generalize some other results from \cite{PS1999} to the case of rational functions.
Assume the equation \eqref{eq:genPellA} is satisfied.

\begin{proposition} Let $\pazocal P\equiv 1$ be the constant polynomial and $\chi_0$ be the $mPE$ extremal rational function with the degree of the denominator equal $g_m$ and the norm $L$. Denote by $x_1^+,\dots, x_{m+g_m}^+$ solutions of the equation $\chi_0(z)=L$ and by $x_1^-,\dots, x_{m+g_m}^-$ solutions of the equation $\chi_0(z)=-L$. Then
\begin{equation}\label{eq:cond1}
\sum_{j=1}^{m+g_m}(x_j^+)^k-\sum_{j=1}^{m+g_m}(x_j^-)^k=0, \quad k\in\{0, 1,\dots, m-1\}.
\end{equation}
\end{proposition}
\begin{proof} It follows the proof of Lemma 2.1 from \cite{PS1999}. Denote $W^+=\chi_0-L$ and $W^-=\chi_0+L$. Then
$$
W^{\pm}=\frac{\prod_{j=1}^{m+g_m}(z-x_j^{\pm})}{\prod_{l=1}^q(z-w_l)}.
$$
 Since the difference $W^+-W^-$ is constant, we have
$$
\frac{\prod_{j=1}^{m+g_m}(z-x_j^{+})}{\prod_{j=1}^{m+g_m}(z-x_j^{-})}=1 + \mathcal{O}\left(\frac{1}{z^m}\right)
$$
  Taking the logarithmic derivative from both sides of the last equality, we get
$$
\frac{(W^+)'}{W^+}-\frac{(W^-)'}{W^-}
=
\mathcal{O}\left(\frac{1}{z^{m+1}}\right)
$$
and
$${\color{black}
\sum_{j=1}^{m+g_m}\frac{1}{z-x_j^+}}
-
\sum_{j=1}^{m+g_m}\frac{1}{z-x_j^-}
=
\sum_{k=0}^\infty
\left[\sum_{j=1}^{m+g_m}(x_j^+)^k-\sum_{j=1}^{m+g_m}(x_j^-)^k\right]
\frac{1}{z^{k+1}
}=
\mathcal{O}\left(\frac{1}{z^{m+1}}\right).
$$
\end{proof}
We will denote by $\sigma_n^k$  symmetric functions of $n$ variables of order $k$ {\color{black} defined as follows:
$$
\sigma_n^k(a_1,\dots,a_n)=\sum_{i_1<\dots<i_k}a_{i_1}\cdots a_{i_k}.
$$
}
.
\begin{proposition}\label{prop:x+-} 
Suppose that given points $x_1^+,\dots, x_{m+g_m}^+$  and $x_1^-,\dots, x_{m+g_m}^-$ satisfy:
\begin{equation}\label{eq:cond2}
\sigma_{m+g_m}^k(x_i^+)= \sigma_{m+g_m}^k(x_i^-),
\quad k\in\{0, 1, \dots, m-1\}.
\end{equation}
Then:
$$
L=\frac{(-1)^{m}}{2}[{\color{black}\sigma_{m+g_m}^m(x_i^+)- \sigma_{m+g_m}^m(x_i^-)}],
$$
$$
h_i=\frac{(-1)^{{\color{black}i+m}}}{2L}[\sigma_{m+g_m}^{i+m}(x_i^+)- \sigma_{m+g_m}^{i+m}(x_i^-)], \, i=1, \dots, g_m,
$$
define the denominator $H_{g_m}(z)=h_{g_m}+h_{{g_m}-1}z+\dots h_1z^{{g_m}-1}+z^{g_m}$ of the extremal rational function
$$
\frac{\prod_{j=1}^{m+{g_m}}(z-x_j^+)}{H_{g_m}(z)}-L=\frac{\prod_{j=1}^{m+{g_m}}(z-x_j^-)}{H_{g_m}(z)}+L=\chi_0.
$$
\end{proposition}

\begin{proof}  Let us first observe that the conditions \eqref{eq:cond1} and \eqref{eq:cond2} are equivalent.
Next, we observe the polynomial identity:
$$
2LH_{g_m}(z)=\prod_{j=1}^{m+{g_m}}(z-x_j^+)-\prod_{j=1}^{m+{g_m}}(z-x_j^-).
$$
The proof now follows from the comparison of the coefficients of the same degree on the left and the right hand side of the polynomial
identity.
\end{proof}

{\color{black}\begin{remark} How can the union $E$ of the closed segments be reconstructed from points $x_j^+$, $x_j^-$ that satisfy the conditions of Proposition \ref{prop:x+-}? 
Note that those points are the solutions of the equations $\chi_0(z)=L$ and $\chi_0(z)=-L$, and they are listed with the corresponding multiplicites.
The points with are listed only once will be the endpoints of the closed intervals which constitute $E$.
\end{remark}
}

\begin{lemma}\label{lemma:alternance} The number of points of an alternance set for the extremal rational function $\chi_0$ is $m+2g_m+1$.
\end{lemma}
\begin{proof} {\color{black} The points of alternanace are the $m+g_m-d$ zeros of $B_{m+g_m-d}$  and $d+1$ endpoints of the intervals as in the polynomial case, where only one endpoint is selected in each of the $d-1$ gap-intervals $(c_{2j-1}, c_{2j-2})$, $j\in\{2, \dots, d\}$. In addition, in the rational case there are also $g_m$ more endpoints of those gap-intervals  $(c_{2j-1}, c_{2j-2})$, where the denominator has zeros, which belong to the alternance set.
These additional $g_m$ points of alternance appear with  each gap interval $(c_{2j-1}, c_{2j-2})$
where the denominator has a zero.  Since the zero is unique in the interval and simple, the values of the function at the endpoints are of the opposite signs. In total, this gives $m+g_m-d+d+1+g_m=m+2g_m+1$ points of alternance.}
\end{proof}

\section{Resonant billiard trajectories within ellipsoid}\label{sec:rresonant}
\label{sec:resonant}

In this section, we will first review the winding numbers associated to periodic trajectories of billiards within ellipsoids in an arbitrary-dimensional Euclidean space.
Then we generalise the frequency map to all trajectories and define resonance for non-periodic ones, and, finally, employ the restricted extremal problem from Section \ref{sec:EPEm} to discuss the caustics of resonant billiard trajectories.


\subsection{Billiards within ellipsoids. Associated winding numbers.}\label{sec:winding}

A variety of higher-dimensional analogues of Poncelet polygons were introduced in \cites{DarbouxSUR, CCS1993, DragRadn2006jmpa, DragRadn2008}, see also \cite{DragRadn2011book} for a systematic exposition and bibliography therein.
Corresponding Cayley-type conditions for periodicity were derived by the present authors in \cites{DragRadn1998a, DragRadn1998b, DragRadn2004, DragRadn2008}, see also \cite{DragRadn2011book}.

Let an ellipsoid in the $d$-dimensional Euclidean space $\mathbf{E}^d$ be given by:
\begin{equation}\label{eq:ellipsoidd}
\E\ :\ \frac{x_1^2}{a_1}+\dots+\frac{x_d^2}{a_d}=1,
\quad
a_1>a_2>\dots>a_d{\color{black}>0}.
\end{equation}
The family of quadrics confocal with $\E$ is:
\begin{equation}\label{eq:confocald}
\Q_{\lambda}\ :\ 
Q_{\lambda}({\color{black}{\bf x}})=\frac {x_1^2}{a_1-\lambda}+\dots + \frac
{x_d^2}{a_d-\lambda}=1,
\end{equation}
with ${\bf x}=(x_1,\dots,x_d)\in\mathbf{E}^d$.

We will say that confocal quadrics $\Q_{\lambda}, \Q_{\mu}$ are \emph{of the same type} if both $\lambda$ and $\mu$ belong to the same open interval among the following ones: $(-\infty,a_d)$, $(a_{j+1},a_j)$, with $j\in\{1,\dots,d-1\}$.
For a point given by its Cartesian coordinates ${\color{black}{\bf x}}=(x_1, \dots, x_d)$, the \emph{Jacobi elliptic coordinates} $(\lambda_1,\dots, \lambda_d)$ are defined as the solutions of the equation in $\lambda$: $ Q_{\lambda}({\bf x})=1$. 
The correspondence between the elliptic and Cartesian coordinates is not injective, since points symmetric with respect to coordinate hyper-planes have equal elliptic coordinates.
Note that the quadrics $\Q_{\lambda_1}, \Q_{\lambda_2}, \dots, \Q_{\lambda_d}$ which contain a given point ${\bf x}$ are orthogonal to each other at the point of intersection and they are all of mutually distinct types
{\color{black}\cite[Appendix 15]{ArnoldMMM}}.

The Chasles theorem  states that almost any line $\ell$ in the space $\mathbf E^d$ is tangent to exactly $d-1$ non-degenerate quadrics from the confocal family {\color{black} and the tangent hyper-planes at the points of tangency are orthogonal to each other {\color{black}\cite[Theorem 2, Appendix 15]{ArnoldMMM}}.
Moreover, if $\ell'$ obtained from $\ell$ by the billiard reflection off a quadric from the confocal family \eqref{eq:confocald}, then those two lines are touching the same $d-1$ confocal quadrics $\Q_{\gamma_1}$, \dots, $\Q_{\gamma_{d-1}}$. 
Moreover, the Jacobi elliptic coordinates $(\lambda_1,\dots, \lambda_d)$ of any point on $\ell$ satisfy the
inequalities $\Pol(\lambda_j)\ge 0$, $j=1,\dots,d$, where
\begin{equation}\label{eq:P}
\Pol(\lambda)=(a_1-\lambda)\dots(a_d-\lambda)(\gamma_1-\lambda)\dots(\gamma_{d-1}-\lambda).
\end{equation}

 \begin{remark}\label{re:Chasles}  Although the Chasles theorem indicates that the tangent hyper-planes to $\Q_{\gamma_1}$, \dots, $\Q_{\gamma_{d-1}}$ at the touching points with $\ell$ are orthogonal to each other, note that 
the points where the line $\ell$ is touching those quadrics are generically distinct.
Thus, those $d-1$ quadrics are not necessarily of different types. More precisely, the possible arrangements of the parameters $\gamma_1$, \dots, $\gamma_{d-1}$ within the intervals $(-\infty,a_d)$, $(a_{j+1},a_j)$,  $j\in\{1,\dots,d-1\}$, can be obtained from the following Lemma \ref{lemma:audin}.
\end{remark}

Let $b_1<b_2<\dots<b_{2d-1}$ be constants such that
$\{b_1,\dots,b_{2d-1}\}=\{a_1,\dots,a_d,\gamma_1,\dots,\gamma_{d-1}\}.$
Clearly, $b_{2d-1}=a_1$.

\begin{lemma}[\cite{Audin1994}]\label{lemma:audin}
If $\gamma_1<\gamma_2< \dots<\gamma_{d-1}$, then
$\gamma_j\in\{b_{2j-1},b_{2j}\}$, for $1\le j\le d-1$.

\end{lemma}

{\color{black}
\begin{example}\label{ex:typescaustics} In three-dimensional space $(d=3)$, according to Lemma \ref{lemma:audin},  there are following possibilities
for types of pairs of quadrics, which touch a generic line $\ell$:
{\color{black}
\begin{itemize}
	\item an ellipsoid and a $1$-sheeted hyperboloid;
	\item an ellipsoid and a $2$-sheeted hyperboloid;
	\item two $1$-sheeted hyperboloids;
	\item a $1$-sheeted and a $2$-sheeted hyperboloid.
\end{itemize}
}
\end{example}
}
If $\ell$ is the line containing a segment of a billiard trajectory within $\E$, then $b_1>0$.

Along a billiard trajectory, the Jacobi elliptic coordinates satisfy:
$$
b_0=0\le\lambda_1\le b_1,
\quad
b_2\le\lambda_2\le b_3,
\quad\dots,\quad
b_{2d-2}\le\lambda_{d}\le b_{2d-1}.
$$
Moreover, along a periodic trajectory, each Jacobi coordinate $\lambda_j$ fills the whole {\color{black} closed interval} {\color{black}$[b_{2j-2},b_{2j-1}]$, with local extreme points being only the endpoints of that} {\color{black}interval}.
Thus, $\lambda_j$ takes values $b_{2j-2}$ and $b_{2j-1}$ alternately and changes monotonically between them.
Let $\T$ be a periodic billiard trajectory and denote by $m_j$ the number of its points where $\lambda_j=b_{2j-2}$.
Based on the previous discussion, $m_j$ is also the number of the points where $\lambda_j=b_{2j-1}$.

Notice that the value $\lambda_1=0$ corresponds to an impact with the boundary ellipsoid $\E$,
value $\lambda_j=\gamma_{k}$ corresponds to a tangency with the caustic $\Q_{\gamma_k}$, and
$\lambda_j=a_k$ corresponds to an intersection with the coordinate hyperplane $x_k=0$.
Since each periodic trajectory intersects any hyperplane even number of times, we get that $m_j$ must be even whenever $b_{2j-2}$ or $b_{2j-1}$ is in the set $\{a_1,\dots, a_d\}$.

\begin{definition}\label{def:windingnumbers}
Following \cite{RR2014}, we denote $m_0=n$ and call
$(m_0, m_1, \dots, m_{d-1})$
\emph{the winding numbers} of a given periodic billiard trajectory with period $n$. In addition we introduce \emph{the elliptic period} $\tilde n$ as the number of impacts after which the trajectory closes in the Jacobi elliptic coordinates.
\end{definition}

\subsection{The frequency map and resonance}

In order to extend the considerations about the winding numbers to the cases of irrational frequencies and non-periodic trajectories, we  employ the potential theory and harmonic analysis, see {\color{black}\cite[Chapter 5]{Si2011}}.
Let us denote

\begin{gather*}
c_1=\frac{1}{b_1},\ \dots,\ c_{2d-1}=\frac{1}{b_{2d-1}},\ c_{2d}=0,
\\
\hat{\mathcal{P}}_{2d}(s)=\prod_{j=1}^{2d}(s-c_j).
\end{gather*}
Consider the curve:
\begin{equation}\label{eq:Curve}
\hat \Curve: t^2=\hat{\mathcal{P}}_{2d}(s).
\end{equation}

 We consider the differential of the third kind
 {\color{black}
 $$\eta=\frac{\eta_{d-1}(s)}{\sqrt{\hat{\mathcal{P}}_{2d}(s)}}\, ds, $$
 where $\eta_{d-1}$ is a  monic polynomial of degree $d-1$
 defined by the conditions:
\begin{equation}\label{eq:diff}
\int_{c_{2j+1}}^{c_{2j}}\frac{\eta_{d-1}(s)}{\sqrt{\hat{\mathcal{P}}_{2d}(s)}}\, ds =0,
\quad
j\in\{d-1, d-2, \dots, 1\},
\end{equation}
The residues of $\eta$ at $\pm\infty$ are $\pm 1$.
Here by $\pm\infty$ we assume the two
points over infinity of the curve \eqref{eq:Curve}. }
 Then the equilibrium measure $\mu$, defined by:
$$
\mu\big([c_{2k},c_{2k-1}]\big)
=
\frac{1}{\pi}
\int_{c_{2k}}^{c_{2k-1}}
\frac{|\eta_{d-1}(s)|}{\sqrt{|\hat{\mathcal{P}}_{2d}(s)|}}\, ds,
$$
induces  a map $\mathbf{m}:\mathcal R^{2d-2}_<  \rightarrow  \mathcal R^{d-1}_+$:
$$
\mathbf{m}: (c_{2d-2}, \dots, c_1) \mapsto \Big(\mu\big([c_{2d-2}, c_{2d-3}]\big),\dots,\mu\big([c_{2}, c_{1}]\big)\Big),
$$
where $\mathcal R^{2d-2}_< $ denotes the set of all finite increasing sequences of $2d-2$ of real numbers.
 The frequency map can be defined as:
$$
F(\gamma_1, \dots, \gamma_{d-1})
=(f_1, \dots, f_{d-1}),
$$
with
$$
f_j=\sum_{k=d-j}^d\mu\big([c_{2k}, c_{2k-1}]\big), \quad j=1,\dots, d-1.
$$
Let us fix a canonical basis {\color{black} $\{\mathbf{a}_1, \dots, \mathbf{a}_{d-1}, \mathbf{b}_1,\dots, \mathbf{b}_{d-1}\}$}
 of homology  on the curve
$\hat \Curve$ {\color{black}in the following way:
\begin{itemize}
\item the projections to the $(t,s)$-plane of cycles $\mathbf{a}_p$, $\mathbf{b}_p$ are clockwise oriented paths surrounding the closed segments $[c_{2(d+1-p)-1}, c_{2(d-p)}]$, $[c_{2d}, c_{2(d+1-p)-1}]$ respectively, for each $1\le p\le d-1$;
\item $\mathbf{b}_p$ and $\mathbf{a}_p$ intersect each other at exactly one point, where the pair of their tangent vectors is a positively oriented basis, for each $1\le p\le d-1$.
\end{itemize}
}
The differentials:
$$
\omega_j=\frac{s^{j-1}}{\sqrt{\hat{\mathcal{P}}_{2d}(s)}}ds, \quad j\in\{1,\dots, d-1\},
$$
form a basis of holomorphic differentials on the curve
$\hat \Curve$. 
Denote by $\eta$ the differential of the third kind, which has zero $\bf a$-periods. 
We denote the {\color{black} scaled} $\bf b$-periods of $\eta$ as:
$$
y_j=\frac{1}{2\pi i}\int_{\mathbf{b}_j}\eta, \quad j\in\{1, \dots, d-1\}.
$$
By using {\color{black} Riemann} Bilinear Relations between the differentials of the first and third kind, see for example {\color{black}\cite[p.~257]{Sp1957}}, one can easily get the following relations:
\begin{equation}\label{eq:frequency}
{\color{black}
\sum_{p=1}^{d-1}y_p\int_{\mathbf{a}_p}\omega_j=2\int_{c_1}^\infty \omega_j, \quad j\in\{1, \dots, d-1\}.
}
\end{equation}

\begin{proposition}
We have:
\begin{itemize}
	\item $|y_j|=f_j$, for all $j\in\{1, \dots, d-1\}$; and
	\item the frequency map is monotonic: $f_1<f_2<\dots<f_{d-1}$.
\end{itemize}
\end{proposition}
\begin{proof}
{\color{black} The validity of $|y_j|=f_j$ follows from the following observations:
\begin{itemize}	
\item[(i)] the polynomial $\eta_{d-1}$ has constant sign on each individual {\color{black}closed interval}
$[c_{2k}, c_{2k-1}]$;
\item[(ii)] on any two {\color{black}consecutive closed intervals} this polynomial has opposite
signs;
\item[(iii)] the analytic extension of the square root
$\sqrt{\hat{\mathcal {P}}_{2d}}$ from one neighbor
{\color{black}interval} to the other changes its sign to opposite.
\end{itemize}

The item (i) follows from the fact  that the polynomial $\eta_{d-1}$ has at least one root in each
one of the $d-1$ gap intervals $(c_{2j+1}, c_{2j})$, which is a consequence of the conditions
\eqref{eq:diff}.  Since the degree of $\eta_{d-1}$ is equal to $d-1$, it follows that all roots of that polynomial
are simple, and there is exactly one root in each of the gap intervals.
}
\end{proof}


Let us denote
\begin{gather*}
\Omega_p
=
\left(\int_{\mathbf{a}_p}\omega_1, \dots, \int_{\mathbf{a}_p}\omega_{d-1}\right)^T, \quad p\in\{1, \dots, d-1\};
\\
\Omega=2\left(\int_{c_1}^\infty\omega_1, \dots, \int_{c_1}^\infty\omega_{d-1}\right)^T.
\end{gather*}
The lattice in $\mathbb R^{d-1}$
$$
\Lambda=\mathbb Z\langle \Omega_1, \dots, \Omega_{d-1}\rangle
$$
is non-degenerate, since the matrix of periods of differentials $\omega_j$ over cycles $\mathbf{a}_p$ is non-degenerate (see, for example~{\color{black}\cite[Chapter 10]{Sp1957}}).
We introduce the following hyperplane in $\mathbb R^d$:
$$
\Pi = \langle \Omega_1, \dots, \Omega_{d-1}, \Omega\rangle,
$$
which is generated by $d-1$ rows of the matrix
$[\Omega_1, \dots, \Omega_{d-1}, \Omega]$.

Then {\color{black} $((-1)^{d-1}y_1, \dots, (-1)^{d-j}y_j, \dots, (-1)y_{d-1}, 1)$} are the Pl\"ucker coordinates of the hyperplane $\Pi$ and they generate the billiard dynamics. 

To the $k$-th iteration of the billiard dynamics, we associate the following:
\begin{itemize}
	\item the vector
	$
	\hat y^{(k)}
	=
	\left(ky_1,\dots, ky_{d-1}\right)$;
	\item the winding numbers $
	\left(m_1^{(k)}, m_2^{(k)}, \dots, m_{d-1}^{(k)}\right)
	=
	\left([|\hat y^{(k)}_1|], [|\hat y^{(k)}_2|], \dots, [|\hat y^{(k)}_{d-1}|]\right)\in \mathbb Z^{d-1}$; and
	\item the residual vector $
	w^{(k)}=\left(w^{(k)}_1, \dots, w^{(k)}_{d-1}\right)
	\in [0, 1)^{d-1},
	$	
\end{itemize}
where
$w^{(k)}_j=\{\hat y^{(k)}_j\}$
 is the fractional part of $|\hat y^{(k)}_j|$, while $[|\hat y^{(k)}_j|]$ is the integer part of $|\hat y^{(k)}_j|$.

\begin{definition}\label{def:r-resonant}
The number of zero components of the vector $w^{(k)}$ will be called \emph{the resonance of the billiard
$k$-path}
and denoted as $r(k)$. The maximal number $r$ among all the numbers $r(k)$ for $k>d$ will be called \emph{the resonance of the billiard trajectory}.
The minimal $k$ such that $r(k)=r$ and $k>d$ will be called \emph{the weak period of the billiard trajectory} and denoted $k_0$.
The associated winding numbers
$\left(m_1^{(k_0)}, m_2^{(k_0)}, \dots, m_{d-1}^{(k_0)}\right)$
are called \emph{the weak winding numbers of the billiard trajectory}.
\end{definition}

\begin{remark} By definition, the resonance of the billiard
$k$-path, the resonance of the billiard trajectory, the weak period of the billiard trajectory, and the weak winding numbers of the billiard trajectory introduced in Definition \ref{def:r-resonant} are the same for all billiard trajectories within $\E$ with fixed caustics
$\Q_{\gamma_1}$, \dots, $\Q_{\gamma_{d-1}}$.
\end{remark}

\begin{proposition}
{\color{black}
Consider a trajectory $\pazocal{T}$ of the billiard within $\E$.
Suppose that $r(k)$ is the resonance of the billiard $k$-path of that trajectory, $r$ the resonance and $k_0$ the weak period.
Then:}
\begin{itemize}
\item[(i)] {\color{black} There are exactly $r(k)$ Jacobi coordinates $\lambda_{\alpha_1}, \dots,\lambda_{\alpha_{r(k)}}$, with the following property:
	along some $k$ consecutive segments of $\pazocal{T}$,
the coordinate $\lambda_{\alpha_j}$ traces its corresponding closed interval $[b_{2\alpha_j-2},b_{2\alpha_j-1}]$ integer number of times.
Moreover, that number of times equals $ky_{\alpha_j}=m_{\alpha_j}^{(k)}$ and the same coordinates will satify the same property along any $k$ consecutive segments of any billiard trajectory having the same caustics as $\pazocal{T}$.}

 In this notation $w_{\alpha_j}^{(k)}=0$, $j\in\{1, \dots, r(k)\}$.
\item[(ii)] The maximal number of rationally dependent among numbers $(y_1, \dots, y_{d-1})$ {\color{black}equals} $r$ and $k_0y_{\alpha_1}$, \dots, $k_0y_{\alpha_r}$ are integers.
\item[(iii)] If {\color{black}$\pazocal{T}$} is periodic with period $n>d$ then $k_0=n$,
$$\left(m_1^{(k_0)}, m_2^{(k_0)}, \dots, m_{d-1}^{(k_0)}\right)=\left(k_0y_1, \dots, k_0y_{d-1}\right)$$
 and the weak winding numbers coincide with winding numbers from Definition \ref{def:windingnumbers}:
    $$\left(m_1^{(k_0)}, m_2^{(k_0)}, \dots, m_{d-1}^{(k_0)}\right)=\left(m_1, m_2, \dots, m_{d-1}\right).$$
    \end{itemize}
\end{proposition}

\begin{proof} According to the Jacobi theorem {\color{black}\cite[Lecture 30]{Jac} and \cite{Dar1870}},
	{along the billiard trajectories with the caustics \color{black}$\Q_{\gamma_1}$, \dots, $\Q_{\gamma_{d-1}}$}  the following differentials {\color{black}equal} zero:
{\color{black}
$$
\sum_{j=1}^d \dfrac{d\lambda_j}{\sqrt{\Pol(\lambda_j)}},
\quad
\sum_{j=1}^d \dfrac{\lambda_j d\lambda_j}{\sqrt{\Pol(\lambda_j)}},\quad
\dots,\quad
\sum_{j=1}^d \dfrac{\lambda_j^{d-2}d\lambda_j}{\sqrt{\Pol(\lambda_j)}},
$$
where $\Pol(\lambda)$ is given by \refeq{eq:P}.
}

Now, we integrate these differentials along $k$ successive segments of a billiard trajectory, rewrite the integrals along the {\color{black} closed intervals} $[b_{2\alpha-2},b_{2\alpha-1}]$ corresponding to the Jacobi coordinates $\lambda_{\alpha}$, for $\alpha\in\{1, \dots, d\}$ and compare with the relations \eqref{eq:frequency}.
This proves (i).

Similarly, integrating the differential along $k_0$ successive segments of the billiard trajectory we get (ii). Finally, for periodic trajectories with period $n$, integrating the differentials along one period of a trajectory, we get (iii).
\end{proof}

\begin{theorem}\label{th:length-periodic}
Suppose that billiard trajectories within $\E$ with caustics $\Q_{\gamma_1}$, \dots, $\Q_{\gamma_{d-1}}$ are periodic. Denote their winding numbers $(m_0,\dots,m_{d-1})$.
Then all those trajectories have the same length $L$, which is equal to:
$$
L=L(\gamma_1, \dots, \gamma_{d-1})=\sum_{j=1}^{d}(-1)^{j+d} m_{j-1}\int_{b_{2j-2}}^{b_{2j-1}}
\frac{\lambda_j^{d-1}d\lambda_j}{+\sqrt{\Pol(\lambda_j)}},
$$
{\color{black}where
$\Pol(\lambda)=(a_1-\lambda)\dots(a_d-\lambda)(\gamma_1-\lambda)\dots(\gamma_{d-1}-\lambda)$.}
\end{theorem}	
\begin{proof}
{\color{black}As Jacobi teaches in \cite[Lecture 30]{Jac}},
the Euclidean length element along the lines tangent to $\Q_{\gamma_1}$, \dots, $\Q_{\gamma_{d-1}}$ {\color{black}is}:
$$
d\ell=\frac12\sum_j\int\frac{\lambda_j^{d-1}d\lambda_j}{\sqrt{\Pol(\lambda_j)}}.
$$
We integrate the last equation along a given periodic billiard trajectory. Using the property of the elliptic coordinates along the periodic billiard trajectory, namely that the elliptic coordinate $\lambda_j$ {\color{black}traces} its  {\color{black}closed} interval $[b_{2j-2}, b_{2j-1}]$ back and forth exactly $m_j$ times along one period of the billiard trajectory. Thus, we get
$$L=\sum_{j=1}^{d}m_{j-1}\int_{b_{2j-2}}^{b_{2j-1}}
\frac{\lambda_j^{d-1}d\lambda_j}{\sqrt{\Pol(\lambda_j)}}.
$$
Note that the signs of one branch of the square root on the closed intervals $[b_{2j-2},b_{2j-1}]$ alternate, thus the statement follows.

We see that the length $L$ does depend only on the caustics, and is equal for all trajectories sharing these caustics:
$L=L(\gamma_1, \dots, \gamma_{d-1})$.
\end{proof}

{\color{black}
Directly using  Theorem 5.15 from \cite{KLN1990}, we get the following:

\begin{theorem}\label{th:r-resonant}
Let $E= \bigcup_{p=1}^{d}[c_{2p}, c_{2p-1}]$. Suppose $k$ is a given integer, $k\ge d$.
 There exists a nonnegative integer $N(k)<d$ and the {\color{black}open} intervals $(c_{2(d+1-\alpha_j)-1}, c_{2(d-\alpha_j)})$, $j\in\{1, \dots, N(k)\}$  with  unique
real numbers $\zeta_j\in (c_{2(d+1-\alpha_j)-1}, c_{2(d-\alpha_j)})$, such that the polynomial $S_{(k)}$ of degree $N(k)$ with the zeros $\zeta_j$ satisfies a generalized Pell's equation
on $E$:
\begin{equation}\label{eq:r-resonant}
S_{(k)}^2=A_{k+N(k)}^2(z)-\prod_{j=1}^{2d}(z-c_j)B_{k+N(k)-d}^2(z).
\end{equation}
Here the polynomials $A_{k+N(k)}$ and $B_{k+N(k)-d}$ are of the degrees $k+N(k)$ and $k+N(k)-d$ respectively.
 \end{theorem}
}

{\color{black}
\begin{definition}\label{def:k-type} For a given union $E= \cup_{p=1}^{d}[c_{2p}, c_{2p-1}]$ of $d$ {\color{black}closed intervals}, we define its \emph{$k$-type} as a $(d-1)$-tuple of zeros and ones,
which has $1$ exactly at $\alpha_j$-th places for $j\in\{1, \dots, N(k)\}$. 
The $k$-type of $E$ records those $N$ out of $d-1$ gap intervals $(c_{2(d+1-\alpha_j)-1}, c_{2(d-\alpha_j)})$ which contain zeros of $S_{(k)}$ from the presentation \eqref{eq:r-resonant}.
\end{definition}
}

{\color{black} Dokaz izbaciti.}

\begin{definition}\label{def:adjoint}
The numbers of zeros of the polynomial $A_{k+N(k)}$ in the {\color{black}closed} intervals
$[c_{2d}, c_2]$, $\dots,$ $[c_{2d}, c_{2d-2}]$ are called
\emph{the adjoint winding numbers
$(\hat m_1^{(k)}, \hat m_2^{(k)}, \dots, \hat m_{d-1}^{(k)})$
of the billiard $k$-path}.  {\color{black} The $k$-type from Definition \ref{def:k-type} is called the adjoint type of the billiard $k$-path}.

Let $N(k_1)$ be the minimal of all $N(k)$, for $k>d$. Then $\hat r=d-1-N(k_1)$ is called \emph{the adjoint
resonance of the billiard trajectory}.

For $k=k_0$ the numbers $\left(\hat m_1^{(k_0)}, \hat m_2^{(k_0)}, \dots, \hat m_{d-1}^{(k_0)}\right)$ are called
\emph{the adjoint winding numbers of the billiard trajectory}.
Denote by $\tau^{(k)}=\left(\tau^{(k)}_1, \dots, \tau^{(k)}_{d-1}\right)$ the numbers of zeros of the polynomial $B_{k+\hat s(k) -d}$ on the {\color{black}open} intervals
 $(c_2, c_1), \dots, (c_{2d}, c_{2d-1})$.
 \end{definition}

\begin{proposition}\label{th:winding}
\begin{itemize}
 \item[\textbf{(a)}] The adjoint winding numbers satisfy:
 $$
 \hat m^{(k)}_{j}=\hat m^{(k)}_{j+1}+\tau^{(k)}_{j}+1,
 \quad
 1\le j\le d-2.
 $$
 \item[\textbf{(b)}] The adjoint winding numbers are strictly decreasing:
 $$
 \hat m^{(k)}_{d-1}<\hat m^{(k)}_{d-2}<\dots<\hat m^{(k)}_1<\hat m^{(k)}_0.
 $$
\end{itemize}
\end{proposition}
\begin{proof}
{\color{black} To simplify the notation, we will use $\hat m$ to denote $\hat m^{(k)}$.}
It follows from {\color{black}\cite[Theorem 5.12]{KLN1990}}  that the number of  zeros of the polynomial $A_{k+N(k)}$ on the {\color{black}open} interval $(c_{2d},c_{2j+1})$ is equal to $\hat m_j$, for
$j\in\{0,\dots,d-1\}$.
The difference $\hat m_{j-1}-\hat m_j$ is thus equal to the number of zeros of the polynomial $A_{k+N(k)}$ on the semi-closed interval $[c_{2j+1},c_{2j-1})$.
Let us recall the mutual order of zeros of polynomials $A_{k+N(k)}$ and $B_{k+N(k) -d}$: they alternate on the {\color{black}closed} interval $[c_{2j+1},c_{2j-1}]$. Thus
the number of zeros of the polynomial $A_{k+N(k)}$ on the {\color{black}closed} interval $[c_{2j+1},c_{2j-1}]$ {\color{black}is bigger by $1$} than the number of zeros of $B_{k+N(k)-d}$ on that interval. This proves (a); (b) follows immediately.
\end{proof}

{\color{black}
\begin{remark}\label{rmk:weakadjoint}
From Theorem \ref{th:r-resonant} (see also
Theorem \ref{the:pell-weak-d} for further discussion) it follows that a given path of $k$-bounces of a given billiard trajectory can be extended
to a periodic billiard trajectory with the same caustics with the period $n=k+N(k)$ with $k$ bounces  off $\E$ and one reflection off each of distinct quadrics $\E_{\mu_j}$, $j\in\{1, \dots, N(k)\}$, in accordance with relation \eqref{eq:r-resonant}: {\color{black} $\zeta_j=\mu_j^{-1}\in
[c_{2(d+1-\alpha_j)-1}, c_{2(d-\alpha_j)}]$. 
The type of the $k$-billiard path determines the types of the quadrics $\E_{\mu_j}$, $j\in\{1, \dots, N(k)\}$.}  In accordance with the generalized Great Poncelet Theorem \cite{CCS1993, DragRadn2011book}, this periodicity property share all the trajectories with the same caustics and boundaries. The order of reflections off different quadrics is irrelevant according to the Double Reflection Theorem, see {\color{black}\cite[Chapter 5]{DragRadn2011book}}.
\end{remark}
Using the terminology of Remark \ref{rmk:weakadjoint} we get
\begin{corollary}\label{cor:extension}
A billiard path of $k>d$ reflections off $\E$ can be extended to a periodic trajectory with the period $n=k+N(k)$ with $k$ bounces  off $\E$ and one reflection off each of distinct quadrics $\E_{\mu_1}$, \dots, $\E_{\mu_{N(k)}}$. {\color{black} The types of the quadrics $\E_{\mu_j}$ are determined by the type of the $k$-billiard path.} The adjoint winding numbers $\left(\hat m_1^{(k)}, \hat m_2^{(k)}, \dots, \hat m_{d-1}^{(k)}\right)$
of the billiard path of length $k$ are the winding numbers of the $n$-periodic trajectory. In particular, the winding number $\hat m_j^{(k)}$,
$j\in\{1, \dots, d-1\}$, represents the number of turns made by the Jacobi coordinate $\lambda_j$ during one period of the $n$-periodic trajectory.
The following inequalities are satisfied between the adjoint winding numbers and the weak winding numbers:
$$
\hat m_j^{(k)}\ge m_j^{(k)}, \quad j\in\{1, \dots, d-1\}.
$$
If the initial path was $k$-periodic trajectory then $m_0=k$ and
$$
\hat m_j^{(k)}= m_j^{(k)}=m_j, \quad j\in\{1, \dots, d-1\}.
$$
\end{corollary}
}

{\color{black}
\begin{remark}
We note that all the extended periodic trajectories from Corollary \ref{cor:extension} have equal length.
That can be proved as in Theorem \ref{th:length-periodic}.
\end{remark}
}

\subsection{Uniqueness of the caustics}
\label{sec:uniquenesscaustics}

The following Lemma reformulates the Audin Alternative, {\color{black} see Lemma \ref{lemma:audin}.  Let us recall that two quadrics  $\Q_{\alpha}$, $\Q_{\beta}$ are of the same type if there exists $j\in\{1, \dots, d-1\}$
such that both parameters $\alpha$, $\beta$ belong to $(a_{j+1},a_j)$ or they are both smaller than $a_d$. (See also Remark \ref{re:Chasles}, Lemma \ref{lemma:audin}, and Example \ref{ex:typescaustics}.)}

\begin{lemma}\label{lemma:same-type-caustics}
If $\Q_{\alpha}$, $\Q_{\beta}$ are caustics of the same type of a given billiard trajectory within $\E$, then
$\{\alpha^{-1}, \beta^{-1}\}=\{c_{2j+1}, c_{2j}\}$, for some $j$.
\end{lemma}
\begin{proof}
According to Lemma \ref{lemma:audin}, exactly one {\color{black} element} of each pair $\{b_{2j-1},b_{2j}\}$ is a parameter of a caustic of the trajectory.
Since $\Q_{\alpha}$, $\Q_{\beta}$ are of the same type, $\alpha$ and $\beta$ must be consecutive in the sequence $b_1,\dots,b_{2n-1}$, so
$\{\alpha,\beta\}=\{b_{2j},b_{2j+1}\}$ for some $j$.
\end{proof}

\begin{theorem}\label{th:signature-caustics}
Let the following be given: an ellipsoid $\E$ in $d$-dimensional space, integers $k_0$, $r$ such that $k_0\ge d$, $0\le r\le d-1$, a collection $\Gamma_1$ of $d-r-1$ quadrics confocal with $\E$, a collection of adjoint winding numbers, and {\color{black} an adjoint type}.
Then, there exists at most one set $\Gamma_2$ of $r$ quadrics confocal with $\E$ which are of the given $r$ types: the billiard trajectories within $\E$, such that their caustics are exactly the quadrics from the set $\Gamma_1\cup\Gamma_2$, have the adjoint resonance equal $r$, the weak period $k_0$, the prescribed adjoint winding numbers, {\color{black} and the adjoint $k_0$-type.}
\end{theorem}

Observe that if one billiard trajectory within $\E$ has $r$ as its adjoint resonance with the weak period $n$, and with the prescribed adjoint winding numbers {\color{black} and the adjoint $k_0$-type}, then all the trajectories within $\E$ with the same caustics share the all listed properties.

\smallskip

\begin{proof}
{\color{black} Suppose the contrary. Let there be at leat two distinct collections $\Gamma_2$ and $\Gamma_2^*$ of $r$ quadrics satisfying the assumptions of the theorem. 
Denote by $c_j$ and $c_j^*$, $j\in\{1,\dots, 2d\}$, the endpoints of the corresponding intervals for $\Gamma_1\cup\Gamma_2$ and $\Gamma_1\cup\Gamma_2^*$ respectively. 
Consider the polynomials $A_{k_0+N(k_0)}, S_{(k_0)}$ from formula \eqref{eq:r-resonant} corresponding to $\Gamma_1\cup\Gamma_2$. Analogous polynomials corresponding to $\Gamma_1\cup\Gamma_2^*$ will be denoted $A_{k_0+N(k_0)}^*, S_{(k_0)}^*$.

The assumption (i) of Theorem \ref{th:pell-unique} is satisfied since $[c_{2d}, c_{2d-1}]=[0, a_d^{-1}]$ and $\Gamma_1\cup\Gamma_2$ and $\Gamma_1\cup\Gamma_2^*$ have $d-r-1$ quadrics in common. {\color{black} To see that the assumption (ii) of Theorem \ref{th:pell-unique} is satisfied, first observe that due to  Lemma \ref{lemma:same-type-caustics} there are no {\color{black}closed intervals} $[c_{2j}, c_{2j-1}]$ with both endpoints being reciprocal to the caustics.  Then from the assumption that corresponding quadrics from $\Gamma_1\cup\Gamma_2$ and $\Gamma_1\cup\Gamma_2^*$ are of the same types, it follows that the coinciding endpoints of the corresponding intervals are both left endpoints or both right endpoints. In other words, it cannot happen that there is $j$ such that $j$-th intervals $[c_{2j}, c_{2j-1}]$ and $[c^*_{2j}, c^*_{2j-1}]$ are such that $c^*_{2j-1}$ is a reciprocal of a caustic and $c_{2j-1}$ is a reciprocal of a pencil parameter.}
{\color{black} The assumptions (iii), (iv), and (v)  of Theorem \ref{th:pell-unique} are also satisfied. This follows from the fact that the adjoint winding numbers and the adjoint $k_0$-type are given}. Thus, we come to contradiction with Theorem \ref{th:pell-unique}, which completes the proof of this statement.}
\end{proof}

\section{Weak periodic billiard trajectories and hyperelliptic curves}\label{sec:sweak}
In Section \ref{sec:rresonant} we investigated
{\color{black}a} classification of non-periodic ellipsoidal
billiard trajectories, based on notions of
{\color{black}
resonance (see Definition \ref{def:r-resonant}) and adjoint resonance (see Definition \ref{def:adjoint}), which depend on dynamical properties of trajectories.
}
Yet another natural classification
of non-periodic ellipsoidal billiard trajectories was introduced by the authors in \cite{DragRadn2008}. This is the concept of weak periodicity, which relays on the geometry of lines tangent to the same set
of quadrics from a given confocal family,
{\color{black}
and will be introduced later in this section, in Definition \ref{def:s-skew-d}.
}

{\color{black}
\begin{definition}\label{def:isospectral-curve}
\emph{The isospectral curve} corresponding to the billiard motion within ellipsoid $\E$ \eqref{eq:ellipsoidd} with the caustics $\Q_{\gamma_1}$, \dots, $\Q_{\gamma_{d-1}}$ from the confocal family \eqref{eq:confocald} is the following hyperelliptic curve:
\begin{equation}\label{eq:hcurve}
\Curve\ :\
y^2=(a_1-x)\dots(a_d-x)(\gamma_1-x)\dots(\gamma_{d-1}-x).
\end{equation}
\end{definition}

{\color{black}We note that the term \emph{isospectral} is used because that curve appeared in the isospectral deformations technique \cite{MV1991}.}

\begin{remark}\label{rem:isospectral}
Note that the equation \eqref{eq:hcurve} is given in $\mathbb{C}^2$ and determines the affine part of the curve $\Curve$.
That affine part is compactified by a single point at the line at the infinity in $\mathbb{CP}^2$.
Such a compactified curve $\Curve'$ will have a singularity at the infinity, and in this paper we will assume that $\Curve$ is obtained from $\Curve'$ by the desingularization.
\end{remark}

We will denote by $P_{b}:=(b,0)$, $P_{\infty}:=(\infty,\infty)$ the Weierstrass points of $\Curve$,}
$$
b\in\{a_1,\dots,a_d,\gamma_1,\dots,\gamma_{d-1}\}.
$$
For a divisor $D$ on the curve, we denote:
\begin{gather*}
\LL(D)=\left\{ f \text{ -- meromorphic function on } \Curve \mid (f)+D\ge 0 \right\},
 \\
\Omega(D)=\left\{ \omega \text{ -- meromorphic differential on } \Curve \mid (\omega)\ge D \right\}.
\end{gather*}
The Riemann-Roch theorem states that
$$
\dim\LL(D)=\deg D- g+\dim\Omega(D)+1,
$$
where $g$ is the genus of the curve.
In our case, $g=d-1$.

We are going now to review the periodicity condition in three-dimensional space.

\subsection{Periodic trajectories in three dimensions}\label{sec:periodic3}
Given an ellipsoid:
\begin{equation}\label{eq:ellipsoid3}
\E\ :\ \frac{x_1^2}{a_1}+\frac{x_2^2}{a_2}+\frac{x_3^2}{a_3}=1,
\quad
a_1>a_2>a_3>0,
\end{equation}
and the confocal family:
\begin{equation}\label{eq:confocal3}
\Q_{\lambda}\ :\ \frac{x_1^2}{a_1-\lambda}+\frac{x_2^2}{a_2-\lambda}+\frac{x_3^2}{a_3-\lambda}=1,
\end{equation}
for which $\Q_{0}=\E$.
According to Chasles' theorem, each trajectory of the billiard within $\E$ has two caustics from the confocal family \refeq{eq:confocal3}.
{\color{black} (See Remark \ref{re:Chasles}, Lemma \ref{lemma:audin}, and Example \ref{ex:typescaustics}.)}
The isospectral curve for trajectories with caustics $\Q_{\gamma_1}$ and $\Q_{\gamma_2}$ is:
\begin{equation}\label{eq:curve3}
\Curve\ :\ y^2=(a_1-x)(a_2-x)(a_3-x)(\gamma_1-x)(\gamma_2-x).
\end{equation}
We denote:
$$
\Pol(x)=(a_1-x)(a_2-x)(a_3-x)(\gamma_1-x)(\gamma_2-x).
$$
For $\mu\in\R$ such that $\Pol(\mu)\ge0$, we denote by $P_{\mu}=P_{\mu}^+$ and $P_{\mu}^-$ the points on $\Curve$ with coordinates $(\mu,+\sqrt{\Pol(\mu)})$ and $(\mu,-\sqrt{\Pol(\mu)})$ respectively.
{\color{black}
We note that, for any $\mu$, the following divisor equivalence holds:
\begin{equation}\label{eq:divisor2P}
P_{\mu}^+ +P_{\mu}^-\sim 2P_{\infty},
\end{equation}
since those are zeros and poles divisors of the function $x-\mu$ on the curve $\Curve$.
}

As previously, we will denote $\{b_1,b_2,b_3,b_4,b_5\}=\{a_1,a_2,a_3,\gamma_1,\gamma_2\}$, with
$b_1<b_2<b_3<b_4<b_5$.
{\color{black}
For the Weierstrass points $P_{b}$, $b\in\{b_1,\dots,b_5\}$, the relation \eqref{eq:divisor2P} becomes:
\begin{equation}\label{eq:divisor-weierstrass}
2P_b\sim 2P_{\infty}.
\end{equation}

In the next Theorem \ref{th:divisor-conditions}, Lemma \ref{lemma:cayley} and Proposition \ref{prop:cayley-double-caustic},
we formulate particular cases from \cite{DragRadn2018},
 for $d=3$.

We will use the symbol $\sim$ to denote the linear equivalence of divisors on the curve $\Curve$.
}

\begin{theorem}[Algebro-geometric conditions for periodicity]
	\label{th:divisor-conditions}
	Consider a billiard trajectory within ellipsoid $\E$,
	with non-degenerate distinct caustics $\Q_{\gamma_1}$ and $\Q_{\gamma_2}$.
	Then the trajectory is $n$-periodic if and only if one of the following is satisfied:
	\begin{itemize}
		\item $n$ is even and $nP_0\sim nP_{\infty}$;
		\item $n$ is even, both caustics are $1$-sheeted hyperboloids and $nP_0\sim(n-2)P_{\infty}+P_{\gamma_1}+P_{\gamma_2}$;
		\item $n$ is odd, $\Q_{\gamma_1}$ is an ellipsoid, and $nP_0\sim(n-1)P_{\infty}+P_{\gamma_1}$.
	\end{itemize}
\end{theorem}

{\color{black}
\begin{remark}
Note that in the first case of Theorem \ref{th:divisor-conditions}, there is no constraint on the type of the caustics.
That means any pair of types of confocal quadrics satisfying the Audin Alternative, Lemma \ref{lemma:audin}, may appear there.
All possible types of caustics for the $3$-dimensional case are listed in Example \ref{ex:typescaustics}.
\end{remark}	
}

\begin{lemma}\label{lemma:cayley}
	Consider a non-singular curve $\Curve$ \refeq{eq:curve3}.
	Then:
	\begin{itemize}
		\item $nP_{0}\sim nP_{\infty}$ for $n$ even if and only if $n\ge6$ and
		$$
		\rank\left(
		\begin{array}{llll}
		A_4 & A_5 &\dots & A_{m+1}\\
		A_5 & A_6 &\dots & A_{m+2}\\
		\dots\\
		A_{m+2}& A_{m+3}&\dots & A_{2m-1}
		\end{array}
		\right)
		<m-2,
		\quad n=2m,
		$$
		with $\sqrt{\Pol(x)}=A_0+A_1x+A_2x^2+\dots$;
		
		\item
		$nP_0\sim (n-2)P_{\infty}+P_{\gamma_1}+P_{\gamma_2}$ for $n$ even if and only if $n\ge4$ and
		$$
		\rank\left(
		\begin{array}{llll}
		B_2 & B_3 &\dots & B_{m}\\
		B_3 & B_4 &\dots & B_{m+1}\\
		\dots\\
		B_{m+1}& B_{m+2}&\dots & B_{2m-1}
		\end{array}
		\right)
		<m-1,
		\quad n=2m,
		$$
		with $\dfrac{\sqrt{\Pol(x)}}{(x-\gamma_1)(x-\gamma_2)}=B_0+B_1x+B_2x^2+\dots$;
		
		\item
		$nP_0\sim (n-1)P_{\infty}+P_{\gamma_1}$ for $n$ odd if and only if $n\ge5$ and
		$$
		\rank\left(
		\begin{array}{llll}
		C_3 & C_4 &\dots & C_{m+1}\\
		C_4 & C_5 &\dots & C_{m+2}\\
		\dots\\
		C_{m+2}& C_{m+3}&\dots & C_{2m}
		\end{array}
		\right)
		<m-1,
		\quad n=2m+1,
		$$
		with $\dfrac{\sqrt{\Pol(x)}}{x-\gamma_1}=C_0+C_1x+C_2x^2+\dots$.
		
	\end{itemize}
\end{lemma}

Next, we will consider the case when the segments of billiard trajectories are placed along generatrices of a confocal $1$-sheeted hyperboloid $\Q_{\gamma_1}$.
In that case, $\Q_{\gamma_1}$ is the unique caustic of that trajectory, which happens when two caustics, both $1$-sheeted hyperboloids ``collide'' with each other, so $\gamma_1=\gamma_2$.
Hyperboloid $\Q_{\gamma_1}$ is then a double caustic of such trajectories.

\begin{proposition}\label{prop:cayley-double-caustic}
	A billiard trajectory within $\E$ with segments on $1$-sheeted hyperboloid $\Q_{\gamma_1}$ is $n$-periodic if and only if $n$ is even and either
	\begin{itemize}
		\item
		$$
		\rank\left(
		\begin{array}{llll}
		A_4 & A_5 &\dots & A_{m+1}\\
		A_5 & A_6 &\dots & A_{m+2}\\
		\dots\\
		A_{m+2}& A_{m+3}&\dots & A_{2m-1}
		\end{array}
		\right)
		<m-2,
		\quad n=2m\ge6,
		$$
		with
		$(\gamma_1-x)\sqrt{(a_1-x)(a_2-x)(a_3-x)}=A_0+A_1x+A_2x^2+\dots$; or
		
		\item
		$$
		\rank\left(
		\begin{array}{llll}
		B_2 & B_3 &\dots & B_{m}\\
		B_3 & B_4 &\dots & B_{m+1}\\
		\dots\\
		B_{m+1}& B_{m+2}&\dots & B_{2m-1}
		\end{array}
		\right)
		<m-1,
		\quad n=2m\ge4,
		$$
		with $\dfrac{\sqrt{(a_1-x)(a_2-x)(a_3-x)}}{\gamma_1-x}=B_0+B_1x+B_2x^2+\dots$.
	\end{itemize}
\end{proposition}

\subsection{$0$-weak periodic trajectories in dimension 3}\label{sec:weak3}

A trajectory $T$ has \emph{$0$-weak period $n$} if
the lines containing
 its first and the $(n+1)$-st segment intersect each other.
In that case, these two segments satisfy the reflection law off one of the quadrics of the confocal family, say $\Q_{\alpha}$,
which is one of three quadrics from the confocal family \eqref{eq:confocal3} containing the intersection point of the segments.

Moreover, any billiard trajectory within $\E$ sharing the same caustics as $T$ will also be of {\color{black} $0$-weak} period $n$, and its first and $(n+1)$-st segment satisfy the billiard reflection law off the same confocal quadric $\Q_{\alpha}$.
In particular, that will also mean that if $\ell$ is any segment of that trajectory and $\ell_{(n+1)}$ a segment obtained after $n$ reflections off the boundary, than they also intersect each other and satisfy the reflection law off $\Q_{\alpha}$.

\begin{lemma}\label{lemma:geom-type}
	If {\color{black} in the above conditions}:
	\begin{itemize}
		\item $\Q_{\alpha}$ is a hyperboloid; or
		\item $\Q_{\alpha}$ is an ellipsoid and $n$ is even
	\end{itemize}
	then the trajectory has a caustic of the same geometric type as $\Q_{\alpha}$.
\end{lemma}
\begin{proof}
	{\color{black}
		Let $M_0M_1\dots M_{n+1}$ be a part of the billiard trajectory within $\E$
	 consisting of $n+1$ consecutive segments.
Since the trajectory is not periodic, i.e.~$M_0\neq M_{n+1}$, that is not a closed polygonal line.  The straight lines containing the first and the last segments, $M_0M_1$ and $M_nM_{n+1}$, intersect at point $M_0'\in\Q_{\alpha}$. They satisfy the reflection law there.
Thus the closed polygon $M_0'M_1M_2\dots M_nM_0'$ represents a closed trajectory of billiard with $n$ reflections off $\E$ and one reflection off $\Q_{\alpha}$.

The Jacobi elliptic coordinates $(\lambda_1,\lambda_2,\lambda_3)$ of any point on the polygon $M_0'M_1M_2\dots M_nM_0'$ satisfy $\lambda_1\le b_1$, $\lambda_2\in[b_2,b_3]$, $\lambda_3\in[b_4,b_5]$.
We note that, along the polygon, each of these coordinates can have a local extremum only in one of the following cases:
\begin{itemize}
	\item at the point of the intersection with a coordinate hyper plane;
	\item at the touching point with a caustic;
	\item at the reflection point.
\end{itemize}
Suppose that $\Q_{\alpha}$ is ellipsoid and $n$ is even.
For any point on $M_1\dots M_n$, its elliptic coordinate $\lambda_1$ belongs to $[0,b_1]$.
The value $\lambda_1=b_1$ is attained exactly once within each segment $M_1M_2$, $M_2M_3$, \dots, $M_{n-1}M_n$, i.e.~$n-1$ times along $M_1\dots M_n$.

If $\alpha<0$, i.e.~the ellipsoid $\Q_{\alpha}$ contains $\E$, then on segments $M_0'M_1$ and $M_nM_0'$, the coordinate $\lambda_1$ belongs to $[\alpha,b_1]$ and it attains the value $\lambda_1=b_1$ also exactly once on each of them.

If $\alpha>0$, i.e.~the ellipsoid $\Q_{\alpha}$ is within $\E$, then on $M_0'M_1$ and $M_nM_0'$, the coordinate $\lambda_1$ belongs to $[0,\alpha]$ and it does not attain the value $\lambda_1=b_1$ there.

We conclude that, on $M_0'M_1M_2\dots M_nM_0'$, the value $\lambda_1=b_1$ is attained either $n-1$ or $n+1$ times, which is an odd number.
Notice that $\lambda_1=a_3$ on the coordinate $x_1x_2$-plane, thus that must happen even number of times on the closed polygon $M_0'M_1M_2\dots M_nM_0'$, so $b_1\neq a_3$, i.e.~$b_1$ is the parameter of one of the caustics of the trajectory.
Since $b_1<a_3$, that caustic is an ellipsoid.

If $\Q_{\alpha}$ is a $1$-sheeted hyperboloid, then, reasoning similarly as in the previous case, we conclude that the elliptic coordinate $\lambda_2$ will, along $M_0'M_1M_2\dots M_nM_0'$, attain one of the endpoints of the {\color{black}closed interval} $[b_2,b_3]$ one time more than the other endpoint, so one of those numbers is odd and the other is even.
Again, the number must correspond to a caustic, since each coordinate hyperplane is crossed even number of times along a closed path.
That caustic will be a $1$-sheeted hyperboloid.

Similarly, when $\Q_{\alpha}$ is a $2$-sheeted hyperboloid, we get that the paramenter $b_4$ will correspond to a caustic, which is also a $2$-sheeted hyperboloid.

}
	
\end{proof}

{\color{black}
	\begin{remark}
Note that each ellipsoid or $1$-sheeted hyperboloid	divide the space into two connected sets, while a $2$-sheeted hyperboloid divides it into three components.

In the case when $\Q_{\alpha}$ is an ellipsoid, one of those sets, the interior of the ellipsoid, is bounded, while the exterior is unbounded.
The interior is in the elliptic coordinates given by $\lambda_1>\alpha$ and the exterior by $\lambda_1<\alpha$.

By analogy, as it was done in \cite{DragRadn2004}, we define \emph{the inside} of a $1$-sheeted hyperboloid $\Q_{\alpha}$
as the set given by $\lambda_2>\alpha$ and its \emph{outside} as given by $\lambda_2<\alpha$.
The inside in this case will be the connected part of the space containing the origin of the coordinate system.

\emph{The inside} of a $2$-sheeted hyperboloid $\Q_{\alpha}$
the set given by $\lambda_3>\alpha$ and its \emph{outside} by $\lambda_3<\alpha$.
The inside in this case will be again the connected part of the space containing the origin of the coordinate system, while the outside consists of the two remaining connected components.

The notions of inside and outside of a confocal quadric are important for the next Proposition \ref{cor:cayley-skew-dim3}, since they determine the side from where the collision with the quadric occurs, which plays role in the calculations.
	\end{remark}
}

\begin{proposition}\label{prop:divisor-skew-dim3}
	Consider a billiard trajectory within $\E$ with caustics $\Q_{\gamma_1}$ and $\Q_{\gamma_2}$.
	The first and the $(n+1)$-st segment of that trajectory intersect
	{\color{black}each other on $\Q_{\alpha}$ and satisfy the billiard law off that quadric}
	 if and only if one of the following conditions is satisfied:
	\begin{itemize}
		\item[(i)] $n$ is odd, $\Q_{\alpha}$ is an ellipsoid, and $$n(P_0-P_{b_1})\pm(P_{\alpha}-P_{b_1})\sim 0;$$
		
		\item[(ii)] $n$ is odd, $\Q_{\alpha}$ is an ellipsoid, both caustics are $1$-sheeted hyperboloids, and $$n(P_0-P_{a_3})\pm(P_{\alpha}-P_{a_3})+P_{\gamma_1}-P_{\gamma_2}\sim0;$$
		
		\item[(iii)] $n$ is odd, $\Q_{\alpha}$ and $\Q_{\gamma_2}$ are hyperboloids of the same type, $\Q_{\gamma_1}$ is ellipsoid, and
		$$n(P_0-P_{\gamma_1})\pm(P_{\alpha}-P_{\gamma_2})\sim 0;$$
		
		\item[(iv)] $n$ is even, $\Q_{\gamma_1}$ is a quadric of the same geometric type as $\Q_{\alpha}$, and $$n(P_0-P_{b_1})\pm(P_{\alpha}-P_{\gamma_1})\sim0.$$		
	\end{itemize}
{\color{black}
In each of the cases,
}	
	 the positive sign corresponds to the reflection off $\Q_{\alpha}$ from inside, and the negative one to the reflection from outside.
\end{proposition}

\begin{proof}
{\color{black}
In the elliptic coordinates, the differential equations
 for lines tangent to $\Q_{\gamma_1}$ and $\Q_{\gamma_2}$
 are:
$$
\begin{aligned}
&\frac{d\lambda_1}{\sqrt{\Pol(\lambda_1)}}
+
\frac{d\lambda_2}{\sqrt{\Pol(\lambda_2)}}
+
\frac{d\lambda_3}{\sqrt{\Pol(\lambda_3)}}
=0,
\\
&\frac{\lambda_1d\lambda_1}{\sqrt{\Pol(\lambda_1)}}
+
\frac{\lambda_2d\lambda_2}{\sqrt{\Pol(\lambda_2)}}
+
\frac{\lambda_3d\lambda_3}{\sqrt{\Pol(\lambda_3)}}
=0,
\end{aligned}
$$
see \cite{Jac}.
Following the ideas of Darboux \cites{Dar1870,DarbouxSUR}, similarly as in \cite[Theorems 1 and 2]{DragRadn2004}, we integrate those equations along the polygonal line $M_0'M_1\dots M_nM_0'$, which is determined as in the proof of Lemma \ref{lemma:same-type-caustics}.
In order to see which divisor relationships will be obtained by such integration, we need to analyse the behaviour of the elliptic coordinates along the integration path.

First, suppose that $\Q_{\alpha}$ is an ellipsoid and that the reflection at point $M_0'$ is from inside that ellipsoid.

Then the elliptic coordinate $\lambda_1$, has values $\lambda_1=\alpha$ at $M_0'$ and $\lambda_1=0$ at $M_1$, \dots, $M_n$.
Along each segment $M_jM_{j+1}$, $1\le j\le n-1$, the coordinate $\lambda_1$ monotonously increases from $0$ to $b_1$ and then decreases from $b_1$ to $0$.
The value $\lambda_1=b_1$ is achieved either at the touching point with the caustic or at the intersection point with the coordinate plane $x_3=0$: the first case occurs when one of the caustics is ellipsoid, and the latter when both caustics are hyperboloids.
Along segments $M_0'M_1$ and $M_0'M_n$, the coordinate $\lambda_1$ monotonously increases from $\alpha$ to $b_1$ and then decreases from $b_1$ to $0$.
Thus, along the polygonal line, the coordinate $\lambda_1$ achieved $n$ times the value $0$, once the value $\alpha$ and $n+1$ times the value $b_1$. Moreover, the values $0$ and $\alpha$ are its local minima and $b_1$ are maxima.
Note that $b_1\in\{\gamma_1,a_3\}$.
If $b_1=a_3$, then the values $\lambda_1=b_1=a_3$ occur at the intersection points with the coordinate plane $x_3=0$, {\color{black} whose quantity}  must be an even number on a closed polygonal line.
Thus, if both caustics are hyperboloids, $n$ must be an odd number.

The coordinates $\lambda_2$ and $\lambda_3$ monotonously change within {\color{black}closed intervals} $[b_2,b_3]$ and $[b_4,b_5]$ respectively, while the only extremal points on the polygon are the endpoints of those {\color{black}intervals}.
Denote by $n_2$ the number of times that $\lambda_2$ achieved each of the values $b_2$, $b_3$, and by $n_3$ the number of times that $\lambda_3$ achieved each of $b_4$, $b_5$.
Since $b_5=a_1$, note that $\lambda_2=b_5$ occurs at the intersection points with the coordinate plane $x_1=0$.
On a closed polygonal line, there must be an even number of such intersections, thus $n_3$ is even.

The differential equations {\color{black} and the classical Abel theorem (stating that two divisors are linearly equivalent if heir images under the Abel map coincide, see e.g.~\cite{Sp1957})} then lead to:
$$
n(P_0-P_{b_1})+(P_{\alpha}-P_{b_1})
+
n_2(P_{b_2}-P_{b_3})
+
n_3(P_{b_4}-P_{b_5})
\sim
0.
$$
Since $n_3$ is even, the equation \eqref{eq:divisor-weierstrass} implies $n_3(P_{b_4}-P_{b_5})\sim0$, so we have:
$$
n(P_0-P_{b_1})+(P_{\alpha}-P_{b_1})
+
n_2(P_{b_2}-P_{b_3})
\sim
0.
$$
If both caustics are $1$-sheeted hyperboloids, i.e.~$b_2=\gamma_1$, $b_3=\gamma_2$, then $b_1=a_3$ and $n$ is odd.
If $n_2$ is even, the term $n_2(P_{\gamma_1}-P_{\gamma_2})$ vanishes beacause of \eqref{eq:divisor-weierstrass}, so we get the case (i) with the plus sign.
If $n_2$ is odd,
$n_2(P_{\gamma_1}-P_{\gamma_1})
\sim
P_{\gamma_1}-P_{\gamma_2},
$
and we get the case (ii) with the plus sign.

If the caustics are hyperboloids of different types, then Lemma \ref{lemma:audin} implies $b_1=a_3$, $b_2=\gamma_1$, $b_3=a_2$, $b_4=\gamma_2$, $b_5=a_1$.
Thus, then $n_2$ must be even, $n$ odd, and we get the case (i) with the plus sign.

If one of the caustics is ellipsoid, we have $b_1=\gamma_1$.
Lemma \ref{lemma:audin} implies $b_2=a_3$, thus $n_2$ is even, so we get cases (i) or (iv) with the plus sign.

Next, suppose that $\Q_{\alpha}$ is an ellipsoid and that the reflection at point $M_0'$ is from outside that ellipsoid.
Then $\Q_{\alpha}$ is within $\E$, i.e.~$\alpha>0$.
Along segments $M_0'M_1$ and $M_0'M_n$, the coordinate $\lambda_1$ monotonously increases from $\alpha$ to $0$.
Thus, along the polygonal line, the coodinate $\lambda_1$ achieved $n$ times the value $0$, once the value $\alpha$ and $n-1$ times the value $b_1$.
In this case, we see that $\alpha$ will be a local maximum for $\lambda_1$.
If the numbers $n_2$, $n_3$ are introduced as in the previous case, we get that the
differential equations yield:
$$
(n-1)(P_0-P_{b_1})+(P_0-P_{\alpha})
+
n_2(P_{b_2}-P_{b_3})
+
n_3(P_{b_4}-P_{b_5})
\sim
0,
$$
or equivalently:
$$
n(P_0-P_{b_1})-(P_{\alpha}-P_{b_1})
+
n_2(P_{b_2}-P_{b_3})
+
n_3(P_{b_4}-P_{b_5})
\sim
0.
$$
Now, analyzing various cases as previously, we will get cases (i), (ii), or (iv), with the minus sign.

The reasoning is similar also when $\Q_{\alpha}$ is a hyperboloid.
}
\end{proof}

\begin{theorem}\label{th:polynomial-skew-dim3}
	A billiard trajectory within $\E$ with caustics $\Q_{\gamma_1}$ and $\Q_{\gamma_2}$ is $0$-weak $n$-periodic if and only if one of the following conditions is satisfied:
	\begin{itemize}
		\item $n=2k+1$ is odd, and there are real polynomials $p_{k+1}$ and $q_{k-2}$ of degrees $k+1$ and $k-2$ such that the polynomial
		$$
		p_{k+1}^2(x)-(a_1-x)(a_2-x)(a_3-x)(\gamma_1-x)(\gamma_2-x)q_{k-2}^2(x)
		$$
		has a zero of order $n$ at $x=0$;
		
		\item $n=2k+1$ is odd, $n\ge 5$, either one caustic is ellipsoid or both are $1$-sheeted hyperboloids, and there are real polynomials $p_{k}$ and $q_{k-1}$ of degrees $k$ and $k-1$ such that the polynomial
		$$
		(\gamma_1-x)(\gamma_2-x)p_{k}^2(x)-(a_1-x)(a_2-x)(a_3-x)q_{k-1}^2(x)
		$$
		has a zero of order $n$ at $x=0$;
		
		\item $n=2k$ is even and there are real polynomials $p_{k}$ and $q_{k-2}$ of degrees $k$ and $k-2$ such that the polynomial
		$$
		(\gamma_1-x)p_{k}^2(x)-(a_1-x)(a_2-x)(a_3-x)(\gamma_2-x)q_{k-2}^2(x)
		$$
		has a zero of order $n$ at $x=0$.
	\end{itemize}
{\color{black}
Moreover, in each of the cases, the corresponding polynomial will also have a simple real root $\alpha<a_3$, such that $\alpha$ is the parameter of the confocal quadric $\Q_{\alpha}$ mentioned in Proposition \ref{prop:divisor-skew-dim3}.
}
\end{theorem}
\begin{proof}
	Suppose that the first and the $(n+1)$-st segments of a trajectory are reflected to each other off quadric $\Q_{\alpha}$.	
	
	Suppose first that (i) from Proposition \ref{prop:divisor-skew-dim3} is satisfied.
	 That divisor relation is equivalent to $nP_0+P_{\alpha}^{\pm}\sim (n+1)P_{\infty}$,
{\color{black}
as a consequence of the relation \eqref{eq:divisor-weierstrass} applied to $P_{b_1}$.}
	
	Since the space $\LL((n+1) P_{\infty})$ is generated by
	$$
	\{1,x,x^2, \dots, x^{k+1},  y, xy, \dots, x^{k-2}y\},
	$$
	{\color{black} where $n=2k+1$,} the relation will be equivalent to the existence of polynomials $p_{k+1}(x)$ and $q_{k-2}(x)$ of degrees $k+1$ and $k-2$ respectively, such that $p_{k+1}(x)+yq_{k-2}(x)$ has a zero of order $n$ at {\color{black}$P_0$} and
	$p_{k+1}(x)\pm yq_{k-2}(x)$ a zero at {\color{black}$P_{\alpha}$}.
	
	Assuming only that $p_{k+1}(x)+yq_{k-2}(x)$ has a zero of order $n$ at {\color{black}$P_0$, we have that the following function also will have a zero of order $n$ at $x=0$:}
	\begin{equation*}
\begin{aligned}
\rho_{n+1}(x)\ &=(p_{k+1}(x)+yq_{k-2}(x))(p_{k+1}(x)-yq_{k-2}(x))\\
&=
p_{k+1}^2(x)-(a_1-x)(a_2-x)(a_3-x)(\gamma_1-x)(\gamma_2-x)q_{k-2}^2(x).
\end{aligned}
\end{equation*}
{\color{black}
Notice that $\rho_{n+1}$ is	a	
	polynomial of degree $n+1$ in $x$ and it has a zero of order $n$ at $x=0$.

Another way to express the same argument is to substitute
$$
y=\sqrt{(a_1-x)(a_2-x)(a_3-x)(\gamma_1-x)(\gamma_2-x)}
$$
to the expression $p_{k+1}(x)+yq_{k-2}(x)$.
In that way, we get a two-valued holomorphic function in $x$, with two holomorphic branches near $x=0$.
One of those branches has a zero of order $n$ at $x=0$, while the other one does not vanish at that point.
Taking the product of the branches, we get the polynomial $\rho_{n+1}(x)$, which then has a zero of order $n$ at $x=0$.
}

We will prove that the remaining zero of $\rho_{n+1}$ lies in the {\color{black}open} interval $(-\infty,a_3)$.
	Notice that $\rho_{n+1}$ is a polynomial of even degree $n+1$ with the positive leading coefficient, thus $\rho_{n+1}(-\infty)>0$. Also, $\rho_{n+1}(a_3)>0$. On the other hand $x=0$ is its zero of order $n$, which means that the polynomial changes the sign at that point.
	That means that $\rho_{n+1}$ will always have another zero $\alpha$, satisfying $\alpha<a_3$.

	The divisor relations (ii) and (iii) from Proposition \ref{prop:divisor-skew-dim3} are both equivalent to
	$$
	nP_0+P_{\alpha}^{\pm}+P_{\gamma_1}+P_{\gamma_2}\sim (n+3)P_{\infty},
	$$
{\color{black}which follows from \eqref{eq:divisor-weierstrass}.
}
	Since the space $\LL((n+3) P_{\infty})$ is generated by
	$$
	\{1,x,x^2, \dots, x^{k+2} y, xy, \dots, x^{k-1}y\},
	$$
	that will be equivalent to the existence of polynomials $p_{k+2}(x)$ and $q_{k-1}(x)$ of degrees $k+2$ and $k-1$ respectfully, such that the function $p_{k+2}(x)+q_{k-1}(x)y$ has a  zero of order $n$ at
	{\color{black}$P_0$} and zeros at
	{\color{black}$P_{\gamma_1}$, $P_{\gamma_2}$}, while the function $p_{k+2}(x)\pm q_{k-1}(x)y$ has a zero at {\color{black}$P_{\alpha}$}.
	We note that
	{\color{black}
	$y$ has zeros at $P_{\gamma_1}$ and $P_{\gamma_2}$},
thus we will have that
	$p_{k+2}(x)=(\gamma_1-x)(\gamma_2-x)p_k(x)$, for a degree $k$ polynomial $p_k$.
	
	Assuming only that $p_{k+2}(x)+q_{k-1}(x)y$ has a zero of order $n$ at
	{\color{black}$P_0$},
	we have that the following
{\color{black}function has  a zero of order $n$ at $x=0$}:
	\begin{equation}\label{eq:rho_n+1}
	\begin{aligned}
	\rho_{n+1}(x)&=\frac{(p_{k+2}(x)+q_{k-1}(x)y)(p_{k+2}(x)-q_{k-1}(x)y)}{(\gamma_1-x)(\gamma_2-x)}
	\\&=
	(\gamma_1-x)(\gamma_2-x)p_{k}^2(x)-(a_1-x)(a_2-x)(a_3-x)q_{k-1}^2(x).
	\end{aligned}
	\end{equation}
{\color{black}
Notice that $\rho_{n+1}$ is a polynomial of degree $n+1$ in $x$ with a zero of order $n$ at $x=0$}.
	Let us examine the position of its remaining zero.
	
{\color{black}	Since} $\rho_{n+1}$ is an even degree polynomial with the positive leading coefficient, {\color{black}we have} $\rho_{n+1}(-\infty)>0$.
	
	In case (ii), we also have $\rho_{n+1}(a_3)>0$.
	On the other hand $x=0$ is a zero of order $n$ of $\rho_{n+1}$, which means that the polynomial changes the sign at that point.
	That implies that $\rho_{n+1}$ will always have another zero $\alpha$, satisfying $\alpha<a_3$.
	
	In case (iii), we have $\gamma_1\in(0,a_3)$, and $\gamma_2$, $\alpha$ are both in $(a_3,a_2)$ or both in $(a_2,a_1)$.
	For $\gamma_2\in(a_3,a_2)$, we have that $\alpha\in(a_3,\gamma_2)$, and we check: $\rho_{n+1}(\gamma_2)<0$, $\rho_{n+1}(a_3)>0$, thus $\rho_{n+1}$ certainly has a zero in the requested interval.
	For $\gamma_2\in(a_2,a_1)$, we have that $\alpha\in(\gamma_2,a_1)$, and we check:
	$\rho_{n+1}(\gamma_2)<0$, $\rho_{n+1}(a_1)>0$, thus again $\rho_{n+1}$ certainly has a zero in the requested interval.
	
	Now, consider the case when $n$ is even, $n=2k$, i.e.~the case (iv) of Proposition \ref{prop:divisor-skew-dim3}.
	The divisor relation from there is equivalent to $nP_0+P_{\alpha}^{\pm}+P_{\gamma_1}\sim (n+2)P_{\infty}$,
	{\color{black}which again follows from the equivalence \eqref{eq:divisor-weierstrass} of double Weierstrass points on hyperelliptic curve}.
	Since the space $\LL((n+2) P_{\infty})$ is generated by
	$\{1,x,x^2, \dots, x^{k+1}, y, xy,\dots, x^{k-2}y\}$, that will be equivalent to the existence of polynomials $p_{k+1}(x)$ and $q_{k-2}(x)$ of degrees $k+1$ and $k-2$ such that the function $p_{k+1}(x)+q_{k-2}(x)y$ has a zero of order $n$ at
	{\color{black}$P_0$} and a zero at {\color{black}$P_{\gamma_1}$}, while the function $p_{k+1}(x)\pm q_{k-2}(x)y$ has a zero at {\color{black}$P_{\alpha}$}.
	We note that $\gamma_1$ is a zero of $y$, thus we will have that
	$p_{k+1}(x)=(\gamma_1-x)p_k(x)$, for a polynomial $p_k$ of degree $k$.
	
	Assuming only that $p_{k+1}(x)+q_{k-2}(x)y$ has a zero of order $n$ at {\color{black}$P_0$}, we have that the following
	{\color{black} function has a zero of order $n$ at $x=0$}:
	\begin{equation}\label{eq:rho_n+1'}
	\begin{aligned}
	\rho_{n+1}(x)&=\frac{(p_{k+1}(x)+q_{k-2}(x)y)(p_{k+1}(x)-q_{k-2}(x)y)}{\gamma_1-x}
	\\&
	=
	(\gamma_1-x)p_k^2(x)-q_{k-2}^2(x)(a_1-x)(a_2-x)(a_3-x)(\gamma_2-x).
	\end{aligned}
	\end{equation}
{\color{black}
Notice that $\rho_{n+1}$ is a polynomial of degree $n+1$, with a zero of order $n$ at $x=0$.
}	
	We will show that the remaining zero of that expression lies in the same interval $(-\infty,a_3)$, $(a_3,a_2)$, $(a_2,a_1)$ as $\gamma_1$.
	
	\emph{Case 1: $\Q_{\gamma_1}$ is ellipsoid.}
	Notice that $\rho_{n+1}$ is an odd degree polynomial with the negative leading coefficient, thus $\rho_{n+1}(-\infty)>0$, and also, $\rho_{n+1}(\gamma_1)<0$.
	From there $\rho_{n+1}$ must have a zero in $(-\infty, \gamma_1)$.
	
	\emph{Case 2: $\Q_{\gamma_1}$ is $1$-sheeted hyperboloid.}
	If $\Q_{\gamma_2}$ is an ellipsoid, then we need to show that $\rho_{n+1}$ has a zero in $(a_3,\gamma_1)$. This will be true since $\rho_{n+1}(\gamma_1)<0$ and $\rho_{n+1}(a_3)>0$.
	
	If $\Q_{\gamma_2}$ is also a $1$-sheeted hyperboloid, then we need to show that $\rho_{n+1}$ has a zero between $\gamma_1$ and $\gamma_2$.
	That will be true since $\sign\rho_{n+1}(\gamma_1)=\sign(\gamma_2-\gamma_1)=-\sign\rho_{n+1}(\gamma_2)$.
	
	If $\Q_{\gamma_2}$ is a $2$-sheeted hyperboloid, then we need to show that $\rho_{n+1}$ has a zero in $(\gamma_1,a_2)$.
	That will be true since $\rho_{n+1}(\gamma_1)>0$ and $\rho_{n+1}(a_2)<0$.
	
	\emph{Case 3: $\Q_{\gamma_1}$ is a $2$-sheeted hyperboloid.}
	Polynomial $\rho_{n+1}$ has a zero in $(\gamma_1,a_1)$ since
	$\rho_{n+1}(\gamma_1)>0$ and $\rho_{n+1}(a_1)<0$.
\end{proof}

\begin{corollary}\label{cor:pell3}
	If the billiard trajectory is $0$-weak $n$-periodic, then there are polynomials $p_{n+1}(x)$, $q_{n-2}(x)$ and $r_1(x)$ of degrees $n+1$, $n-2$ and $1$ such that:
	\begin{equation}\label{eq:weak-pell}
	p_{n+1}^2(x)
	-
	\Pol(x)
	q_{n-2}^2(x)
	=
	x^{2n} r_1^2(x).
	\end{equation}
\end{corollary}

\begin{proof}
	According to Theorem \ref{th:polynomial-skew-dim3}, there are three possible cases.
	
	In the first case of Theorem \ref{th:polynomial-skew-dim3}, there are polynomials $p_{k+1}$ and $q_{k-2}$ of degrees $k+1$ and $k-2$ such that
	$$
	p_{k+1}^2(x)-(a_1-x)(a_2-x)(a_3-x)(\gamma_1-x)(\gamma_2-x)q_{k-2}^2(x)=x^n(x-\alpha).
	$$
	Set:
	$$
	p_{n+1}(x):=p_{k+1}^2(x)-\frac{1}2x^n\left(x-\alpha\right).
	$$
	We have:
	$$
	\begin{aligned}
	p_{n+1}^2(x)
	\ &=
	p_{k+1}^2(x)
	\left(
	p_{k+1}^2(x)-x^n(x-{\alpha})
	\right)
	+
	\frac14 x^{2n} \left(x-\alpha\right)^2
	\\
	&=
	p_{k+1}^2(x)
	\cdot
	\Pol(x) q_{k-2}^2(x)
	+
	\frac{1}4 x^{2n} \left(x-{\alpha}\right)^2.
	\end{aligned}
	$$
	Thus:
	$$
	q_{n-2}=p_{k+1} q_{k-2},
	\quad
	r_1(x)=\frac{1}2\left(x-{\alpha}\right).
	$$
	
	In the second case of Theorem \ref{th:polynomial-skew-dim3}, there are polynomials $p_{k}$ and $q_{k-1}$ of degrees $k$ and $k-1$ such that
	$$
	(\gamma_1-x)(\gamma_2-x)p_{k}^2(x)-(a_1-x)(a_2-x)(a_3-x)q_{k-1}^2(x)=x^n(x-\alpha).
	$$
	Set:
	$$
	p_{n+1}(x):=
	({\gamma_1}-x)({\gamma_2}-x)p_{k}^2(x)
	-
	\frac{1}2x^n\left(x-{\alpha}\right).
	$$
	We have:
	$$
	\begin{aligned}
	p_{n+1}^2(x)
	\ &=
	({\gamma_1}-x)({\gamma_2}-x)p_{k}^2(x)
	\left(
	({\gamma_1}-x)({\gamma_2}-x)p_{k}^2(x)
	-x^n\left(x-{\alpha}\right)
	\right)
	+
	\frac{1}4x^{2n}\left(x-{\alpha}\right)^2
	\\
	&=
	p_{k}^2(x)
	\cdot
	\Pol(x) q_{k-1}^2(x)
	+
	\frac{1}4 x^{2n} \left(x-{\alpha}\right)^2.
	\end{aligned}
	$$
	Thus:
	$$
	q_{n-2}= p_{k} q_{k-1},
	\quad
	r_1(x)=\frac{1}2\left(x-{\alpha}\right).
	$$
	
	In the third case of Theorem \ref{th:polynomial-skew-dim3}, there are polynomials $p_{k}$ and $q_{k-2}$ of degrees $k$ and $k-2$ such that
	$$
	(\gamma_1-x)p_{k}^2(x)-(a_1-x)(a_2-x)(a_3-x)(\gamma_2-x)q_{k-2}^2(x)=-x^n(x-\alpha).
	$$
	Notice the negative sign at the righthand side, since the leading coefficient of the lefthand side is negative.
	
	Set:
	$$
	p_{n+1}(x):=
	\left({\gamma_1}-x\right) p_{k}^2(x)
	+
	\frac{1}2 x^n \left(x-{\alpha}\right).
	$$
	We have:
	$$
	\begin{aligned}
	p_{n+1}^2(x)
	\ &=
	(\gamma_1-x)\tilde p_{k}^2(x)
	\left(
	(\gamma_1-x) p_{k}^2(x)+x^n (x-\alpha)
	\right)
	+
	\frac14 x^{2n} \left(x-\alpha\right)^2
	\\
	&=
	p_{k}^2(x)
	\cdot
	\Pol(x) q_{k-2}^2(x)
	+
	\frac14 x^{2n} \left(x-{\alpha}\right)^2.
	\end{aligned}
	$$
	Thus:
	$$
	q_{n-2}= p_{k} q_{k-2},
	\quad
	r_1(x)=\frac12\left(x-{\alpha}\right).
	$$
\end{proof}

\begin{corollary}\label{cor:cayley-skew-dim3}
	A billiard trajectory within $\E$ with caustics $\Q_{\gamma_1}$ and $\Q_{\gamma_2}$ is $0$-weak $n$-periodic if and only if one of the following conditions is satisfied:
	\begin{itemize}
		
				\item $n=2k+1$ is odd, $n\ge 5$, and
		\begin{equation}\label{eq:detA}
		\det\left(
		\begin{array}{cccc}
		A_4 & A_5 & \dots & A_{k+2}\\
		A_5 & A_6 & \dots & A_{k+3}\\
		\dots\\
		A_{k+2} & A_{k+3}&\dots& A_{2k}
		\end{array}
		\right)
		=0;
		\end{equation}

		\item $n=2k+1$ is odd, either one of the caustics is ellipsoid or both of them are $1$-sheeted hyperboloids, and $\det M=0$, where $M$ is $n\times n$ matrix with entries:
		\begin{equation}\label{eq:Mij}
		M_{ij}=
		\begin{cases}
		\gamma_1\gamma_2, & i+j=n-k+1;
		\\
		-\gamma_1-\gamma_2, & i+j=n-k+2,\ j\le n-k;
		\\
		1, & i+j=n-k+3,\ j\le n-k;
		\\
		A_{i-1-j+n}, & j\ge n-k+1,\ i\ge j-n+1;
		\\
		0, &\text{otherwise};
		\end{cases}
		\end{equation}

		\item $n=2k$ is even and $\det N=0$, where $N$ is $n\times n$ matrix with entries:
		\begin{equation}\label{eq:Nij}
		N_{ij}=
		\begin{cases}
		-\gamma_1, & i+j=n-k+2;
		\\
		1, & i+j=n-k+3,\ j\le n-k+1;
		\\
		A_{i-1-j+n}, & j\ge n-k,\ i\ge j-n+1;
		\\
		0, &\text{otherwise}.
		\end{cases}
		\end{equation}
	\end{itemize}
Coefficients $A_j$ are defined as:
$\sqrt{\Pol(x)}=A_0+A_1x+A_2x^2+\dots$.
\end{corollary}
\begin{proof}
{\color{black}
In the proof of the first condition of Theorem \ref{th:polynomial-skew-dim3} is that $n=2k+1$ is odd and that there are real polynomials:
$$
p_{k+1}(x)=\sum_{j=0}^{k+1}\alpha_j x^j,
\quad
q_{k-2}(x)=\sum_{j=0}^{k-2}\beta_j x^j,
$$
such that the function $f=p_{k+1}(x)+y q_{k-2}(x)$, which is defined on $\Curve$, has a zero of order $n$ at $P_0$.
Thus, the condition is equivalent to vanishing of the first $n$ coefficients in the Taylor series of $f$ around $P_0$, which gives a system of $n$ homogeneous linear equations in $n$ unknowns $\alpha_0$, \dots, $\alpha_{k+1}$, $\beta_0$, \dots, $\beta_{k-2}$.
The matrix of that system has a $(k+2)\times(k+2)$ identity submatrix in the upper left corner, thus the existence of its non-zero solution is equivalent to the determinant of its $(k-1)\times(k-1)$ submatrix in the lower right corner being equal to zero, i.e.~to \ref{eq:detA}.

The determinant conditions for the remaining cases are obtained in a similar way.
}	
\end{proof}

\subsection{$0$-weak periodicity in elliptic coordinates in dimension 3}\label{sec:weak3elliptic}

In this section, we will consider billiard trajectories in elliptic coordinates, that is up to the reflections with respect to the coordinate planes.
Equivalently, this is the billiard within one of the solids bounded by the ellipsoid and the three coordinate planes.

{\color{black}
Let us explain that more precisely.
Suppose that $\pazocal{T}$ is a billiard trajectory within $\E$.
Trajectory $\pazocal{T}$ \emph{in elliptic coordinates} is its projection $\pazocal{T}_e$ which maps any point $X$ with Cartesian coordinates $(x_1,x_2,x_3)$ to the point {\color{black} $X_e=(|x_1|,|x_2|,|x_3|)$}.
We note that $X_e$ is the unique point in the first octant of the space which has the same elliptic coordinate as $X$.

\begin{definition}\label{def:0-weak-elliptic}
We say that a billiard trajectory $\pazocal{T}$ within $\E$ is \emph{$0$-weak $n$-periodic in elliptic coordinates} if
the projections $t_1^e$, $t_{n+1}^e$ of the first and $(n+1)$-st segments of $\pazocal{T}$ intersect on a confocal quadric $\Q_{\alpha}$ and satisfy the billiard reflection law off that quadric at the intersection point.
\end{definition}}

\begin{lemma}\label{lemma:divisor-skew3-elliptic}
{\color{black}A} billiard trajectory within $\E$ with caustics $\Q_{\gamma_1}$ and $\Q_{\gamma_2}$
{\color{black}$0$-weak $n$-periodic in elliptic coordinates}
 if and only if, for some $j\in\{1,2,3,4,5\}$ such that $\Pol(x)\ge0$ for each $x$ between $\alpha$ and $b_j$ and some $\varepsilon_1,\varepsilon_2\in\{0,1\}$,
	the following identity holds:
	\begin{equation}\label{eq:divisor-skew3-elliptic}
	n(P_0-P_{b_1})
	\pm(P_{\alpha}-P_{b_j})
	+\varepsilon_1(P_{b_2}-P_{b_3})
	+\varepsilon_2(P_{b_4}-P_{b_5})
	\sim 0.
	\end{equation}
\end{lemma}
{\color{black}
\begin{proof}
This can be proved similarly as Proposition \ref{prop:divisor-skew-dim3}.
The key difference is that here, for periodicity in elliptic coordinates, it is not true that the number of intersections of a polygonal line with a coordinate hyperplane must be even.
Thus the numbers $n_2$, $n_3$ introduced as in the proof of Proposition \ref{prop:divisor-skew-dim3} can be even or odd in any of the cases.
Using \eqref{eq:divisor-weierstrass}, we get
$n_2(P_{b_2}-P_{b_3})\sim\varepsilon_1(P_{b_2}-P_{b_3})$ and $n_3(P_{b_4}-P_{b_5})\sim\varepsilon_2(P_{b_4}-P_{b_5})$,
with $\varepsilon_1$, $\varepsilon_2$ are of the same parity as $n_2$, $n_3$ respectively.
\end{proof}	
}

\subsubsection*{Polynomial equations in dimension 3}

We will derive polynomial equations for $0$-{\color{black}weak} periodic trajectories {\color{black}in elliptic coordinates}.

Suppose first that $\Q_{\alpha}$ is ellipsoid and $n$ even.
Then \refeq{eq:divisor-skew3-elliptic} yields
$$
nP_0+P_{\alpha}^{\pm}+P_{b_1}
+\varepsilon_1(P_{b_2}+P_{b_3})
+\varepsilon_2(P_{b_4}+P_{b_5})
\sim (n+2+2\varepsilon_1+2\varepsilon_2)P_{\infty}.
$$
Since the space $\LL(2mP_{\infty})$, $m=\dfrac{n}{2}+1+\varepsilon_1+\varepsilon_2$, is generated by:
$$
1,x,\dots,x^{m},
y,xy,\dots,x^{m-3}y,
$$
the divisor identity is equivalent to the existence of real polynomials $p_m(x)$ and $q_{m-3}(x)$, such that {\color{black}the function} $p_m(x)+yq_{m-3}(x)$
{\color{black}
on} $\Curve$ has a zero of order $n$ at
{\color{black}
$P_0$}, a zero at {\color{black}$P_{b_1}$}, zeros of order $\varepsilon_1$ at {\color{black}$P_{b_2}$, $P_{b_3}$}, and zeros of order $\varepsilon_2$ at
{\color{black}
$P_{b_4}$, $P_{b_5}$}, while $p_m(x)\pm yq_{m-3}(x)$ has a zero at {\color{black}$P_{\alpha}$}.
Since $b_1$, \dots, $b_5$ are zeros of $y$, we have that
$p_m(x)
=
(b_1-x)(b_2-x)^{\varepsilon_1}(b_3-x)^{\varepsilon_1}
(b_4-x)^{\varepsilon_2}(b_5-x)^{\varepsilon_2}
p_{m-1-2\varepsilon_1-2\varepsilon_2}(x)$,
and we conclude that the following expression:
$$
\begin{aligned}
\rho_{n+1}(x)=\ &
\frac{(p_m(x)+yq_{m-3}(x))(p_m(x)-yq_{m-3}(x))}
{(b_1-x)(b_2-x)^{\varepsilon_1}(b_3-x)^{\varepsilon_1}
	(b_4-x)^{\varepsilon_2}(b_5-x)^{\varepsilon_2}}
\\
=\ &
(b_1-x)(b_2-x)^{\varepsilon_1}(b_3-x)^{\varepsilon_1}
(b_4-x)^{\varepsilon_2}(b_5-x)^{\varepsilon_2}
p_{m-1-2\varepsilon_1-2\varepsilon_2}^2(x)
\\
&-
(b_2-x)^{1-\varepsilon_1}(b_3-x)^{1-\varepsilon_1}
(b_4-x)^{1-\varepsilon_2}(b_5-x)^{1-\varepsilon_2}
q_{m-3}^2(x)
\end{aligned}
$$
is a polynomial of degree $n+1$ with a zero of order $n$ at $x=0$ and a simple zero at $x=\alpha$.
The leading coefficient of $\rho_{n+1}$ is negative, so $\rho_{n+1}(-\infty)>0$.
We also have that $\rho_{n+1}(b_1)<0$.
Notice that at $x=0$ the polynomial will not change the sign, so $\alpha\in(-\infty,b_1)$.

Second, suppose that $\Q_{\alpha}$ is ellipsoid and $n$ is odd.
Then \refeq{eq:divisor-skew3-elliptic} yields
$$
nP_0+P_{\alpha}^{\pm}
+\varepsilon_1(P_{b_2}+P_{b_3})
+\varepsilon_2(P_{b_4}+P_{b_5})
\sim (n+1+2\varepsilon_1+2\varepsilon_2)P_{\infty}.
$$
Since the space $\LL(2mP_{\infty})$, $m=\dfrac{n+1}{2}+\varepsilon_1+\varepsilon_2$, is generated by:
$$
1,x,\dots,x^{m},
y,xy,\dots,x^{m-3}y,
$$
the divisor identity is equivalent to the existence of real polynomials $p_m(x)$ and $q_{m-3}(x)$, such that $p_m(x)+yq_{m-3}(x)$ has a zero of order $n$ at {\color{black}$P_0$}, zeros of order $\varepsilon_1$ at {\color{black}$P_{b_2}$, $P_{b_3}$}, and zeros of order $\varepsilon_2$ at {\color{black}$P_{b_4}$, $P_{b_5}$}, while $p_m(x)\pm yq_{m-3}(x)$ has a zero at
{\color{black}
$P_{\alpha}$}.
Since $b_1$, \dots, $b_5$ are zeros of $y$, we have that
$p_m(x)
=
(b_2-x)^{\varepsilon_1}(b_3-x)^{\varepsilon_1}
(b_4-x)^{\varepsilon_2}(b_5-x)^{\varepsilon_2}
p_{m-2\varepsilon_1-2\varepsilon_2}(x)$,
and we conclude that the following expression:
$$
\begin{aligned}
\rho_{n+1}(x)=\ &
\frac{(p_m(x)+yq_{m-3}(x))(p_m(x)-yq_{m-3}(x))}
{(b_2-x)^{\varepsilon_1}(b_3-x)^{\varepsilon_1}
	(b_4-x)^{\varepsilon_2}(b_5-x)^{\varepsilon_2}}
\\
=\ &
(b_2-x)^{\varepsilon_1}(b_3-x)^{\varepsilon_1}
(b_4-x)^{\varepsilon_2}(b_5-x)^{\varepsilon_2}
p_{m-2\varepsilon_1-2\varepsilon_2}^2(x)
\\
&-
(b_1-x)(b_2-x)^{1-\varepsilon_1}(b_3-x)^{1-\varepsilon_1}
(b_4-x)^{1-\varepsilon_2}(b_5-x)^{1-\varepsilon_2}
q_{m-3}^2(x)
\end{aligned}
$$
is a polynomial of degree $n+1$ with a zero of order $n$ at {\color{black}$P_0$} and a simple zero at {\color{black}$P_{\alpha}$}.
The leading coefficient of $\rho_{n+1}$ is positive, so $\rho_{n+1}(-\infty)>0$.
We also have that $\rho_{n+1}(b_1)>0$.
Notice that at $x=0$ the polynomial will change the sign, so $\alpha\in(-\infty,b_1)$.

Third, suppose that $\Q_{\alpha}$ is $1$-sheeted hyperboloid and $n$ even.
Then \refeq{eq:divisor-skew3-elliptic} yields
$$
nP_0+P_{\alpha}^{\pm}
+\varepsilon_1 P_{b_2}+(1-\varepsilon_1) P_{b_3}
+\varepsilon_2(P_{b_4}+P_{b_5})
\sim (n+2+2\varepsilon_2)P_{\infty}.
$$
Since the space $\LL(2mP_{\infty})$, $m=\dfrac{n}{2}+1+\varepsilon_2$, is generated by:
$$
1,x,\dots,x^{m},
y,xy,\dots,x^{m-3}y,
$$
the divisor identity is equivalent to the existence of real polynomials $p_m(x)$ and $q_{m-3}(x)$, such that $p_m(x)+yq_{m-3}(x)$ has a zero of order $n$ at
{\color{black}
$P_0$}, zeros of order $\varepsilon_1$, $1-\varepsilon_1$ at {\color{black}$P_{b_2}$, $P_{b_3}$}, and zeros of order $\varepsilon_2$ at
{\color{black}
$P_{b_4}$, $P_{b_5}$}, while $p_m(x)\pm yq_{m-3}(x)$ has a zero at {\color{black}$P_{\alpha}$}.
Since $b_1$, \dots, $b_5$ are zeros of $y$, we have that
$p_m(x)
=
(b_2-x)^{\varepsilon_1}(b_3-x)^{1-\varepsilon_1}
(b_4-x)^{\varepsilon_2}(b_5-x)^{\varepsilon_2}
p_{m-1-2\varepsilon_2}(x)$,
and we conclude that the following expression:
$$
\begin{aligned}
\rho_{n+1}(x)=\ &
\frac{(p_m(x)+yq_{m-3}(x))(p_m(x)-yq_{m-3}(x))}
{(b_2-x)^{\varepsilon_1}(b_3-x)^{1-\varepsilon_1}
	(b_4-x)^{\varepsilon_2}(b_5-x)^{\varepsilon_2}}
\\
=\ &
(b_2-x)^{\varepsilon_1}(b_3-x)^{1-\varepsilon_1}
(b_4-x)^{\varepsilon_2}(b_5-x)^{\varepsilon_2}
p_{m-1-2\varepsilon_2}^2(x)
\\
&-
(b_1-x)(b_2-x)^{1-\varepsilon_1}(b_3-x)^{\varepsilon_1}
(b_4-x)^{1-\varepsilon_2}(b_5-x)^{1-\varepsilon_2}
q_{m-3}^2(x)
\end{aligned}
$$
is a polynomial of degree $n+1$ with a zero of order $n$ at $x=0$ and a simple zero at $x=\alpha$.
It is easy to see that $\rho_{n+1}(b_2)>0$ and $\rho_{n+1}(b_3)<0$, thus $\alpha\in(b_2,b_3)$.

Fourth, suppose that $\Q_{\alpha}$ is $1$-sheeted hyperboloid and $n$ odd.
Then \refeq{eq:divisor-skew3-elliptic} yields
$$
nP_0+P_{\alpha}^{\pm}+P_{b_1}
+\varepsilon_1 P_{b_2}+(1-\varepsilon_1) P_{b_3}
+\varepsilon_2(P_{b_4}+P_{b_5})
\sim (n+3+2\varepsilon_2)P_{\infty}.
$$
Since the space $\LL(2mP_{\infty})$, $m=\dfrac{n+1}{2}+1+\varepsilon_2$, is generated by:
$$
1,x,\dots,x^{m},
y,xy,\dots,x^{m-3}y,
$$
the divisor identity is equivalent to the existence of real polynomials $p_m(x)$ and $q_{m-3}(x)$, such that $p_m(x)+yq_{m-3}(x)$ has a zero of order $n$ at
{\color{black}
$P_0$}, simple zero at {\color{black}$P_{b_1}$}, zeros of order $\varepsilon_1$, $1-\varepsilon_1$ at
{\color{black}
$P_{b_2}$, $P_{b_3}$}, and zeros of order $\varepsilon_2$ at
{\color{black}
$P_{b_4}$, $P_{b_5}$}, while $p_m(x)\pm yq_{m-3}(x)$ has a zero at
{\color{black}
$P_{\alpha}$}.
Since $b_1$, \dots, $b_5$ are zeros of $y$, we have that
$p_m(x)
=
(b_1-x)(b_2-x)^{\varepsilon_1}(b_3-x)^{1-\varepsilon_1}
(b_4-x)^{\varepsilon_2}(b_5-x)^{\varepsilon_2}
p_{m-2-2\varepsilon_2}(x)$,
and we conclude that the following expression:
$$
\begin{aligned}
\rho_{n+1}(x)=\ &
\frac{(p_m(x)+yq_{m-3}(x))(p_m(x)-yq_{m-3}(x))}
{(b_1-x)(b_2-x)^{\varepsilon_1}(b_3-x)^{1-\varepsilon_1}
	(b_4-x)^{\varepsilon_2}(b_5-x)^{\varepsilon_2}}
\\
=\ &
(b_1-x)(b_2-x)^{\varepsilon_1}(b_3-x)^{1-\varepsilon_1}
(b_4-x)^{\varepsilon_2}(b_5-x)^{\varepsilon_2}
p_{m-2-2\varepsilon_2}^2(x)
\\
&-
(b_2-x)^{1-\varepsilon_1}(b_3-x)^{\varepsilon_1}
(b_4-x)^{1-\varepsilon_2}(b_5-x)^{1-\varepsilon_2}
q_{m-3}^2(x)
\end{aligned}
$$
is a polynomial of degree $n+1$ with a zero of order $n$ at $x=0$ and a simple zero at $x=\alpha$.
It is easy to see that $\rho_{n+1}(b_2)<0$ and $\rho_{n+1}(b_3)>0$, thus $\alpha\in(b_2,b_3)$.

Fifth, suppose that $\Q_{\alpha}$ is $2$-sheeted hyperboloid and $n$ even.
Then \refeq{eq:divisor-skew3-elliptic} yields
$$
nP_0+P_{\alpha}^{\pm}
+\varepsilon_1 (P_{b_2}+ P_{b_3})
+\varepsilon_2 P_{b_4}+ (1-\varepsilon_2)P_{b_5}
\sim (n+2+2\varepsilon_2)P_{\infty}.
$$
Since the space $\LL(2mP_{\infty})$, $m=\dfrac{n}{2}+1+\varepsilon_2$, is generated by:
$$
1,x,\dots,x^{m},
y,xy,\dots,x^{m-3}y,
$$
the divisor identity is equivalent to the existence of real polynomials $p_m(x)$ and $q_{m-3}(x)$, such that $p_m(x)+yq_{m-3}(x)$ has a zero of order $n$ at
{\color{black}
$P_0$}, zeros of order $\varepsilon_1$ at
{\color{black}
$P_{b_2}$, $P_{b_3}$}, and zero of order $\varepsilon_2$, $1-\varepsilon_2$ at
{\color{black}
$P_{b_4}$, $P_{b_5}$}, while $p_m(x)\pm yq_{m-3}(x)$ has a zero at
{\color{black}
$P_{\alpha}$}.
Since $b_1$, \dots, $b_5$ are zeros of $y$, we have that
$p_m(x)
=
(b_2-x)^{\varepsilon_1}(b_3-x)^{\varepsilon_1}
(b_4-x)^{\varepsilon_2}(b_5-x)^{1-\varepsilon_2}
p_{m-1-2\varepsilon_2}(x)$,
and we conclude that the following expression:
$$
\begin{aligned}
\rho_{n+1}(x)=\ &
\frac{(p_m(x)+yq_{m-3}(x))(p_m(x)-yq_{m-3}(x))}
{(b_2-x)^{\varepsilon_1}(b_3-x)^{\varepsilon_1}
	(b_4-x)^{\varepsilon_2}(b_5-x)^{1-\varepsilon_2}}
\\
=\ &
(b_2-x)^{\varepsilon_1}(b_3-x)^{\varepsilon_1}
(b_4-x)^{\varepsilon_2}(b_5-x)^{1-\varepsilon_2}
p_{m-1-2\varepsilon_2}^2(x)
\\
&-
(b_1-x)(b_2-x)^{1-\varepsilon_1}(b_3-x)^{1-\varepsilon_1}
(b_4-x)^{1-\varepsilon_2}(b_5-x)^{\varepsilon_2}
q_{m-3}^2(x)
\end{aligned}
$$
is a polynomial of degree $n+1$ with a zero of order $n$ at $x=0$ and a simple zero at $x=\alpha$.
It is easy to see that $\rho_{n+1}(b_4)>0$ and $\rho_{n+1}(b_5)<0$, thus $\alpha\in(b_2,b_3)$.

Sixth, suppose that $\Q_{\alpha}$ is $1$-sheeted hyperboloid and $n$ odd.
Then \refeq{eq:divisor-skew3-elliptic} yields
$$
nP_0+P_{\alpha}^{\pm}+P_{b_1}
+\varepsilon_1 (P_{b_2}+P_{b_3})
+\varepsilon_2P_{b_4}+(1-\varepsilon_2)P_{b_5}
\sim (n+3+2\varepsilon_2)P_{\infty}.
$$
Since the space $\LL(2mP_{\infty})$, $m=\dfrac{n+1}{2}+1+\varepsilon_2$, is generated by:
$$
1,x,\dots,x^{m},
y,xy,\dots,x^{m-3}y,
$$
the divisor identity is equivalent to the existence of real polynomials $p_m(x)$ and $q_{m-3}(x)$, such that $p_m(x)+yq_{m-3}(x)$ has a zero of order $n$ at
{\color{black}
$P_0$}, simple zero at
{\color{black}
$P_{b_1}$}, zeros of order $\varepsilon_1$ at
{\color{black}
$P_{b_2}$, $P_{b_3}$}, and zeros of order $\varepsilon_2$, $1-\varepsilon_2$ at
{\color{black}
$P_{b_4}$, $P_{b_5}$}, while $p_m(x)\pm yq_{m-3}(x)$ has a zero at
{\color{black}
$P_{\alpha}$}.
Since $b_1$, \dots, $b_5$ are zeros of $y$, we have that
$p_m(x)
=
(b_1-x)(b_2-x)^{\varepsilon_1}(b_3-x)^{\varepsilon_1}
(b_4-x)^{\varepsilon_2}(b_5-x)^{1-\varepsilon_2}
p_{m-2-2\varepsilon_2}(x)$,
and we conclude that the following expression:
$$
\begin{aligned}
\rho_{n+1}(x)=\ &
\frac{(p_m(x)+yq_{m-3}(x))(p_m(x)-yq_{m-3}(x))}
{(b_1-x)(b_2-x)^{\varepsilon_1}(b_3-x)^{\varepsilon_1}
	(b_4-x)^{\varepsilon_2}(b_5-x)^{1-\varepsilon_2}}
\\
=\ &
(b_1-x)(b_2-x)^{\varepsilon_1}(b_3-x)^{\varepsilon_1}
(b_4-x)^{\varepsilon_2}(b_5-x)^{1-\varepsilon_2}
p_{m-2-2\varepsilon_2}^2(x)
\\
&-
(b_2-x)^{1-\varepsilon_1}(b_3-x)^{1-\varepsilon_1}
(b_4-x)^{1-\varepsilon_2}(b_5-x)^{\varepsilon_2}
q_{m-3}^2(x)
\end{aligned}
$$
is a polynomial of degree $n+1$ with a zero of order $n$ at $x=0$ and a simple zero at $x=\alpha$.
It is easy to see that $\rho_{n+1}(b_4)<0$ and $\rho_{n+1}(b_5)>0$, thus $\alpha\in(b_4,b_5)$.

\begin{theorem}\label{the:pell-weak-3}
	The billiard trajectory is $0$-{\color{black}weak} $n$-periodic in elliptic coordinates if and only if there are polynomials $p_{n+1}(x)$, $q_{n-2}(x)$ and $r_1(x)$ of degrees $n+1$, $n-2$ and $1$ such that \refeq{eq:weak-pell} is satisfied.
\end{theorem}
\begin{proof}
	If the trajectory is $0$-{\color{black}weak} $n$-periodic, we can derive \refeq{eq:weak-pell} similarly as in Corollary \ref{cor:pell3}, from the polynomial equations derived in this section.
	
	Now, we suppose that there there are polynomials $p_{n+1}(x)$,
	$q_{n-2}(x)$ and $r_1(x)$ satisfying \refeq{eq:weak-pell}.
	Then:
	$$
	(p_{n+1}(x)-x^n r_1(x))(p_{n+1}(x)+x^n r_1(x))=\Pol(x) q_{n-2}^2(x).
	$$
	Notice that $p_{n+1}(x)$ and $x^nr_1(x)$ do not have common factors, which implies the same for $p_{n+1}(x)-x^n r_1(x)$ and $p_{n+1}(x)+x^n r_1(x)$.
	Thus, we get:
	$$
	p_{n+1}(x)-x^n r_1(x)=\rho_1(x) \kappa_1^2(x),
	\quad
	p_{n+1}(x)+x^n r_1(x)=\rho_2(x) \kappa_2^2(x),
	$$
	with
	$$
	\rho_1\rho_2=\Pol,
	\quad
	\kappa_1\kappa_2= q_{n-2}.
	$$
	Subtracting the two equalities, we get:
	$$
	\rho_2(x)\kappa_2^2(x)-\rho_1(x)\kappa_1^2(x)=2 x^n r_1.
	$$
	The possible distributions of factors of $\Pol(x)$ into $\rho_1$ and $\rho_2$ are obtained using $\Pol(\alpha)>0$ for $r_1(\alpha)=0$ and the fact that the polynomial $\rho_2(x)\kappa_2^2(x)-\rho_1(x)\kappa_1^2(x)=2 x^n r_1$ has only $2$ zeros: $x=0$, of multiplicity $n$, and a simple zero $x=\alpha$.
\end{proof}

\subsection{Weak periodic trajectories in dimension $d$ and generalized Pell's equations}\label{sec:sweakdPell}

The consideration of the previous Sections  can be generalized as follows, using the concept of $s$-skew lines from \cite{DragRadn2008}, see also \cite{DragRadn2011book}. For two given lines {\color{black} $\ell_1$ and $\ell_2$} which share the same caustics in the $d$-dimensional space, we say that they
are $s$-skew if $s$ is the smallest number
{\color{black}no greater than $d-2$
 such that there exists a system of
$s+1$} quadrics $\Q_{\alpha_k}$, $k = 1, \dots, s+1$ from the confocal family, such that
the line {\color{black} $\ell_2$ is obtained from $\ell_1$} by consecutive reflections off  $\Q_{\alpha_k}$. If {\color{black}two distinct} lines {\color{black} $\ell_1$ and $\ell_2$} intersect {\color{black}or they are parallel}, they are $0$-skew. {\color{black}The lines} are $(-1)$-skew if they coincide.

{\color{black}
Note that some of the quadrics from the system $\Q_{\alpha_k}$ can coincide among themselves or with $\E$.

Let us explain why the inequality $s\le d-2$ is imposed in the notion of $s$-skew lines.
Namely, any pair of the lines in the space is $s$-skew for the unique value {\color{black} $s\in\{-1,0,\dots,d-2\}$} and, moreover, the system of quadrics $\Q_{\alpha_k}$ is uniquely determined, up to their order \cite{DragRadn2008}.
If the requirement $s\le d-2$ would be omitted, then the uniqueness would be lost.

}

\begin{definition}\label{def:s-skew-d} A billiard trajectory is \emph{$s$-weak $n$-periodic} if
{\color{black}the lines containing}	
	 its first and $(n+1)$-st segments are $s$-skew.
\end{definition}

\begin{theorem}\label{the:pell-weak-d}
	The billiard trajectory is $s$-weak periodic with period $n$ in elliptic coordinates if and only if there are polynomials $p_{n+s+1}(x)$, $q_{n+s+1-d}(x)$ and $r_{s+1}(x)$ of degrees $n+s+1$, $n+s+1-d$ and $s+1$ such that they pairwise do not have any common factors and
\begin{equation}\label{eq:pell-weak-d}
	p_{n+s+1}^2(x)-\Pol_{2d-1}(x)q_{n+s+1-d}^2(x)=x^{2n}r_{s+1}^2(x).
\end{equation}
\end{theorem}
\begin{proof}
A trajectory will be $s$-weak $n$-periodic in elliptic coordinates if and only the following relation is satisfied:
\begin{equation}\label{eq:skew-periodic-divisor}
n(P_0-P_{b_1})
+
\sum_{k=1}^{s+1}(P_{\alpha_k}^{\sigma_k}-P_{b_{j_k}})
+
\sum_{l=1}^{d-1}\varepsilon_l(P_{b_{2l}}-P_{b_{2l+1}})
\sim
0,
\end{equation}
with $\varepsilon_l\in\{0,1\}$, and $\sigma_k\in\{+,-\}$.
Here, $j_k\in\{1,\dots,2d-1\}$ and $\Pol_{2d-1}(x)>0$ for all $x$ between $b_{j_k}$ and $\alpha_k$.
The relation \refeq{eq:skew-periodic-divisor} is equivalent to the equation of the following form:
\begin{equation}\label{eq:skew-periodic-divisor2}
nP_0
+
\sum_{k=1}^{s+1}P_{\alpha_k}^{\sigma_k}
+
\sum_{l=1}^{2d-1}\varepsilon_l'P_{b_l}
\sim
\left(n+s+1+\sum_{l=1}^{2d-1}\varepsilon_l'\right)P_{\infty},
\end{equation}
with $\varepsilon_l'\in\{0,1\}$.
We note that the order of the divisors on both sides of the equivalence \refeq{eq:skew-periodic-divisor2} is even, thus denote $2m=n+s+1+\sum_{l=1}^{2d-1}\varepsilon_l'$.

Then $\LL(2mP_{\infty})$ is generated by:
$$
1,x,\dots,x^m,xy,\dots,x^{m-d}y,
$$
and \refeq{eq:skew-periodic-divisor2} is equivalent to the existence of polynomials $p_m(x)$ and $q_{m-d}(x)$ such that $p_m(x)+yq_{m-3}(x)$ has a zero of order $n$ at
{\color{black}
the point $P_0$ on $\Curve$}, zeros of orders $\varepsilon_l'$ at {\color{black}the Weirstrass points
$P_{b_l}$}, while
$p_m(x)+\sigma_k yq_{m-3}(x)$ has a zero at {\color{black}$P_{\alpha_k}$}.
Since $y$ {\color{black}has zeros at $P_{b_l}$}, we have that:
$$
p_m(x)=\prod_{l=1}^{2d-1}(b_{l}-x)^{\varepsilon_l'}p_{m_1}(x).
$$
From there:
\begin{equation}\label{eq:rho_n+s+1}
\begin{aligned}
\rho_{n+s+1}(x)
\ &=
\frac
{(p_m(x)+yq_{m-3}(x))(p_m(x)-yq_{m-3}(x))}
{\prod_{l=1}^{2d-1}(b_{l}-x)^{\varepsilon_l'}}
\\
&=
\prod_{l=1}^{2d-1}(b_{l}-x)^{\varepsilon_l'}p_{m_1}^2(x)
-
\prod_{l=1}^{2d-1}(b_{l}-x)^{1-\varepsilon_l'}q_{m-3}^2(x)
\end{aligned}
\end{equation}
is a polynomial {\color{black} in $x$} of degree $n+s+1$ with a zero of order $n$ at $x=0$ and zeros at $x=\alpha_k$, $1\le k\le s+1$, thus
$\rho_{n+s+1}(x)=c x^n\prod_{k=1}^{s+1}(\alpha_k-s)$,
with $c=(-1)^{\sum\varepsilon_l'+s+1}$.

We set:
$$
p_{n+s+1}(x)
:=
\prod_{l=1}^{2d-1}(b_{l}-x)^{\varepsilon_l'}p_{m_1}^2(x)
+
\frac{c}2\rho_{n+s+1}(x).
$$
We have:
$$
\begin{aligned}
p_{n+s+1}^2(x)
\ &=
\prod_{l=1}^{2d-1}(b_{l}-x)^{\varepsilon_l'}p_{m_1}^2(x)
\left(
\prod_{l=1}^{2d-1}(b_{l}-x)^{\varepsilon_l'}p_{m_1}^2(x)
+
c \rho_{n+s+1}(x)
\right)
+
\frac14 \rho_{n+s+1}^2(x)
\\
&=
p_{m_1}^2(x)\cdot\Pol_{2d-1}(x)q_{m-3}^2(x)
+
\frac14\rho_{n+s+1}^2(x).
\end{aligned}
$$
Thus
$$
q_{n+s+1-d}=p_{m_1}q_{m-3},
\quad
r_{s+1}=\frac12\prod_{k=1}^{s+1}(\alpha_k-x).
$$

Now, suppose that there are polynomials $p_{n+s+1}$, $q_{n+s+1-d}$, $r_{s+1}$ satisfying \refeq{eq:pell-weak-d}.
Then:
$$
(p_{n+s+1}(x)-x^nr_{s+1})
(p_{n+s+1}(x)+x^nr_{s+1})
=
\Pol_{2d-1}(x)q_{n+s+1-d}^2.
$$
Since those polynomials do not have pairwise any common factors, we have:
\begin{gather*}
p_{n+s+1}(x)-x^nr_{s+1}=\rho_1(x)\kappa_1^2(x),
\quad
p_{n+s+1}(x)+x^nr_{s+1}=\rho_2(x)\kappa_2^2(x),
\\
\rho_1\rho_2=\Pol_{2d-1},
\quad
\kappa_1\kappa_2=q_{n+s+1-d}.
\end{gather*}
From there:
$$
\rho_2(x)\kappa_2^2(x)-\rho_1\kappa_1^2(x)
=
2x^nr_{s+1}(x),
$$
which is equivalent to \refeq{eq:rho_n+s+1} and will lead back to the divisor condition and the $s$-skew periodicity of the corresponding trajectories.
\end{proof}

Now we are ready to relate the notion of adjoint resonance of billiard trajectories from  Section \ref{sec:resonant} and $s$-weak periodic from the current Section, see Definitions \ref{def:r-resonant} and \ref{def:s-skew-d}.
\begin{theorem}\label{the:resonant-weak-d} If a billiard trajectory within an ellipsoid in $d$-dimensional space is {\color{black} of adjoint resonance $\hat r$} and $s$-weak periodic, then
\begin{equation}\label{eq:resonanat-skew}
{\color{black} \hat r}+s+2\le d.
\end{equation}
\end{theorem}
\begin{proof}
The proof follows from Theorems \ref{th:r-resonant} and \ref{the:pell-weak-d}. One should compare
the equations \eqref{eq:r-resonant} and \eqref{eq:pell-weak-d}, after making the change of variables in the latter $z=1/x$. The inequality in the condition is a consequence of the condition  of Theorem \ref{th:r-resonant} that the polynomial $S_{(k)}$ has at most one zero in each of the {\color{black}closed} intervals $[b_{2(d+1-i)-1}, b_{2(d-i)}]$ and no zeros outside the union of the intervals. There is no analogue of such condition in Theorem \ref{the:pell-weak-d}, see the examples in the following sections.
\end{proof}

\subsection{Examples in dimension $3$}\label{sec:examplesdim3}

\paragraph*{$0$-weak period $3$.}
Suppose that the first and fourth segments of a trajectory
{\color{black} intersect each other on $\Q_{\alpha}$ and satisfy the reflection law off that quadric}.

\emph{Case 1: $\Q_{\alpha}$ is an ellipsoid.} Then one of the following relations must be satisfied:
\begin{itemize}
	\item $3(P_0-P_{b_1})\pm(P_{\alpha}-P_{b_1})\sim0$; or
	\item $3(P_0-P_{a_3})\pm(P_{\alpha}-P_{a_3})+P_{\gamma_1}-P_{\gamma_2}\sim0$, with $\Q_{\gamma_1}$ and $\Q_{\gamma_2}$ both being $1$-sheeted hyperboloids,
\end{itemize}
where the positive sign corresponds to the reflection off $\Q_{\alpha}$ from inside, and the negative one to the reflection from outside.

The first relation is equivalent to $3P_0+P_{\alpha}^{\pm}\sim 4P_{\infty}$.
Since the space $\LL(4 P_{\infty})$ is generated by $\{1,x,x^2\}$, the relation will be equivalent to the existence of a non-trivial quadratic polynomial in $x$ with a triple zero at $x=0$ and a zero at $x=\alpha$, which is not possible.

The second relation is equivalent to
$3P_0+P_{\alpha}^{\pm}+P_{\gamma_1}+P_{\gamma_2}\sim 6P_{\infty}$.
Since the space $\LL(6 P_{\infty})$ is generated by $\{1,x,x^2, x^3, y\}$, that will be equivalent to the existence of a cubic polynomial $p_3(x)$ such that the function $p_3(x)+y$ has a triple zero at $x=0$ and zeros at $x=\gamma_1$, $x=\gamma_2$, while the function $p_3(x)\pm y$ has a zero at $x=\alpha$.
We note that $\gamma_1$ and $\gamma_2$ are zeros of $y$, thus we will have that
$p_3(x)=(\gamma_1-x)(\gamma_2-x)p_1(x)$, for a linear polynomial.
If $p_1(x)=ax+b$, we have:
$$
(\gamma_1-x)(\gamma_2-x)p_1(x)+y=A_0 + b \gamma_1 \gamma_2 + (A_1 - b \gamma_1 - b \gamma_2 + a \gamma_1 \gamma_2) x + (A_2 + b - a \gamma_1 -
a \gamma_2) x^2+\dots,
$$
with $A_0,A_1,A_2,\dots$ defined as in Lemma \ref{lemma:cayley}.
Thus the expression will have a triple zero at $x=0$ for some choice of $a$ and $b$ if and only if:
\begin{equation}\label{eq:det3}
\det\left(
\begin{matrix}
0 & \gamma_1\gamma_2  & A_0\\
\gamma_1\gamma_2 & -\gamma_1-\gamma_2 &  A_1\\
-\gamma_1-\gamma_2 & 1 & A_2
\end{matrix}
\right)
=0.
\end{equation}

Assuming that these conditions are satisfied, we have that the following polynomial of degree $4$ has a triple zero at $x=0$:
\begin{equation}\label{eq:rho}
\rho(x)=\frac{(p_3(x)+y)(p_3(x)-y)}{(\gamma_1-x)(\gamma_2-x)}
=
(\gamma_1-x)(\gamma_2-x)p_1^2(x)-(a_1-x)(a_2-x)(a_3-x),
\end{equation}
and we need it to have also a zero at $x=\alpha$, for some $\alpha<a_3$.

Notice that $\rho$ is a quartic polynomial with the positive leading coefficient, thus $\rho(-\infty)>0$. Also, $\rho(a_3)>0$. On the other hand $x=0$ is a triple zero of $\rho$, which means that the polynomial changes the sign at that point.
That means that $\rho$ will always have another zero $\alpha$, satisfying $\alpha<a_3$.

\emph{Case 2: $\Q_{\alpha}$ is a hyperboloid.}
According to Lemma \ref{lemma:geom-type}, one of the caustics is a hyperboloid of the same geometric type as $\Q_{\alpha}$.
Since the number of bounces off the boundary $\E$ is odd, the other caustic is an ellipsoid.
Thus, we assume $\gamma_1\in(0,a_3)$, and $\gamma_2$, $\alpha$ are both in $(a_3,a_2)$ or both in $(a_2,a_1)$.
The condition for weak $3$-periodicity then is:
$$
3(P_0-P_{\gamma_1})\pm(P_{\alpha}-P_{\gamma_2})\sim0,
$$
where the positive sign corresponds to the reflection off $\Q_{\alpha}$ from inside, and the negative one to the reflection from outside.
That relation is equivalent to
$3P_0+P_{\alpha}^{\pm}+P_{\gamma_1}+P_{\gamma_2}\sim6P_{\infty}$.
Similarly as above, we get that the condition for weak $3$-periodicity is equivalent to the relation \refeq{eq:det3} and the existence of a linear polynomial $p_1$ such that \refeq{eq:rho} has a zero within the corresponding interval.
For $\gamma_2\in(a_3,a_2)$, we have that $\alpha\in(a_3,\gamma_2)$, and we check: $\rho(\gamma_2)<0$, $\rho(a_3)>0$, thus $\rho$ certainly has a zero in the requested interval.
For $\gamma_2\in(a_2,a_1)$, we have that $\alpha\in(\gamma_2,a_1)$, and we check:
$\rho(\gamma_2)<0$, $\rho(a_1)>0$, thus again $\rho$ certainly has a zero in the requested interval.

We conclude that the condition for weak $3$-periodicity is equivalent to the following:
\begin{itemize}
	\item the relations \refeq{eq:det3} are satisfied; and
	\item either one of the caustics is ellipsoid or both caustics are $1$-sheeted hyperboloids.
\end{itemize}

\paragraph*{$0$-weak period $4$.}
Suppose that the first and fifth segments of a trajectory are reflected to each other off quadric $\Q_{\alpha}$.

According to Lemma \ref{lemma:geom-type}, one of the caustics, say $\Q_{\gamma_1}$, is of the same geometric type as $\Q_{\alpha}$.
The following relation must be satisfied:
$$
4(P_0-P_{b_1})\pm(P_{\alpha}-P_{\gamma_1})\sim0,
$$
where the positive sign corresponds to the reflection off $\Q_{\alpha}$ from inside, and the negative one to the reflection from outside.
This is equivalent to $4P_0+P_{\alpha}^{\pm}+P_{\gamma_1}\sim 6P_{\infty}$.
Since the space $\LL(6 P_{\infty})$ is generated by $\{1,x,x^2, x^3, y\}$, that will be equivalent to the existence of a cubic polynomial $p_3(x)$ such that the function $p_3(x)+y$ has a zero of order $4$ at $x=0$ and a zero at $x=\gamma_1$, while the function $p_3(x)\pm y$ has a zero at $x=\alpha$.
We note that $\gamma_1$ is a zero of $y$, thus we will have that
$p_3(x)=(\gamma_1-x)p_2(x)$, for a quadratic polynomial $p_2(x)=ax^2+bx+c$.
We have that the following expression:
$$
(\gamma_1-x)p_2(x)+y
=A_0 + c \gamma_1 + (A_1 - c + b \gamma_1) x + (A_2 - b + a \gamma_1) x^2 + (-a + A_3) x^3+\dots
$$
will have a zero of order $3$ at $x=0$ for some choice of $a$, $b$, $c$ if and only if:
\begin{equation}\label{eq:det4}
\det\left(
\begin{matrix}
0 & 0  & -\gamma_1  & A_0\\
0 & -\gamma_1 & 1 &  A_1\\
-\gamma_1 & 1 & 0 & A_2\\
1 & 0 & 0 & A_3
\end{matrix}
\right)
=0.
\end{equation}
with $A_0,A_1,A_2,\dots$ defined as in Lemma \ref{lemma:cayley}.

Assuming that this condition is satisfied, we have that the following polynomial of degree $5$ has a zero of order $4$ at $x=0$:
\begin{equation}\label{eq:rho5}
\rho_5(x)=\frac{(p_3(x)+y)(p_3(x)-y)}{\gamma_1-x}
=
(\gamma_1-x)p_2^2(x)-(a_1-x)(a_2-x)(a_3-x)(\gamma_2-x),
\end{equation}
and we need that it also has a zero at $x=\alpha$, $\alpha$ lying in the same {\color{black}open} interval $(-\infty,a_3)$, $(a_3,a_2)$, $(a_2,a_1)$ as $\gamma_1$.

\emph{Case 1: $\Q_{\alpha}$ and $\Q_{\gamma_1}$ are ellipsoids.}
Notice that $\rho_5$ is a quintic polynomial with the negative leading coefficient, thus $\rho_5(-\infty)>0$, and also, $\rho_5(\gamma_1)<0$.
From there $\rho_5$ must have a zero in $(-\infty, \gamma_1)$.

\emph{Case 2: $\Q_{\alpha}$ and $\Q_{\gamma_1}$ are $1$-sheeted hyperboloids.}
If $\Q_{\gamma_2}$ is an ellipsoid, then we need to show that $\rho_5$ has a zero in $(a_3,\gamma_1)$. This will be true since $\rho_5(\gamma_1)<0$ and $\rho_5(a_3)>0$.

If $\Q_{\gamma_2}$ is a $1$-sheeted hyperboloid, then we need to show that $\rho_5$ has a zero between $\gamma_1$ and $\gamma_2$.
That will be true since $\sign\rho_5(\gamma_1)=\sign(\gamma_2-\gamma_1)=-\sign\rho_5(\gamma_2)$.

If $\Q_{\gamma_2}$ is a $2$-sheeted hyperboloid, then we need to show that $\rho_5$ has a zero in $(\gamma_1,a_2)$.
That will be true since $\rho_5(\gamma_1)>0$ and $\rho_5(a_2)<0$.

\emph{Case 3: $\Q_{\alpha}$ and $\Q_{\gamma_1}$ are $2$-sheeted hyperboloids.}
Polynomial $\rho_5$ has a zero in $(\gamma_1,a_1)$ since
$\rho_5(\gamma_1)>0$ and $\rho_5(a_1)<0$.

We conclude that  the condition for weak $4$-periodicity is equivalent to \refeq{eq:det4}.

\paragraph*{$0$-weak period $5$.}
Suppose that the first and sixth segments of a trajectory are reflected to each other off quadric $\Q_{\alpha}$.

\emph{Case 1: $\Q_{\alpha}$ is an ellipsoid.} Then one of the following relations must be satisfied:
\begin{itemize}
	\item $5(P_0-P_{b_1})\pm(P_{\alpha}-P_{b_1})\sim0$; or
	\item $5(P_0-P_{a_3})\pm(P_{\alpha}-P_{a_3})+P_{\gamma_1}-P_{\gamma_2}\sim0$, with $\Q_{\gamma_1}$ and $\Q_{\gamma_2}$ both being $1$-sheeted hyperboloids,
\end{itemize}
where the positive sign corresponds to the reflection off $\Q_{\alpha}$ from inside, and the negative one to the reflection from outside.

The first relation is equivalent to $5P_0+P_{\alpha}^{\pm}\sim 6P_{\infty}$.
Since the space $\LL(6 P_{\infty})$ is generated by $\{1,x,x^2, x^3,y\}$, the relation will be equivalent to the existence of a cubic polynomial $p_3(x)$ such that $p_3(x)+y$ has a zero of order $5$ at $x=0$ and $p_3(x)\pm y$ a zero at $x=\alpha$.
The expression $p_3(x)+y$ will have a zero of order $5$ at $x=0$ for some $p_3$ if and only if
$A_4=0$, where $A_4$ is defined as in Lemma \ref{lemma:cayley}.
Assuming that this condition is satisfied, we have that the following polynomial of degree $6$ has a zero of order $5$ at $x=0$:
\begin{equation*}
\rho_6(x)=(p_3(x)+y)(p_3(x)-y)
=
p_3^2(x)-(a_1-x)(a_2-x)(a_3-x)(\gamma_1-x)(\gamma_2-x),
\end{equation*}
and we need that it also has a zero at $x=\alpha$, $\alpha$ lying in the {\color{black}open} interval $(-\infty,a_3)$.
Notice that $\rho_6$ is a polynomial of degree $6$ with the positive leading coefficient, thus $\rho_6(-\infty)>0$. Also, $\rho_6(a_3)>0$. On the other hand $x=0$ is its zero of order $5$, which means that the polynomial changes the sign at that point.
That means that $\rho$ will always have another zero $\alpha$, satisfying $\alpha<a_3$.

The second relation is equivalent to
$5P_0+P_{\alpha}^{\pm}+P_{\gamma_1}+P_{\gamma_2}\sim 8P_{\infty}$.
Since the space $\LL(8 P_{\infty})$ is generated by
$\{1,x,x^2, x^3, x^4, y, xy\}$,
that will be equivalent to the existence of polynomials $p_4(x)$ and $q_1(x)$ of degrees $4$ and $1$ respectfully, such that the function $p_4(x)+q_1(x)y$ has a  zero of order $5$ at
{\color{black}
$P_0$} and zeros at
{\color{black}
$P_{\gamma_1}$, $P_{\gamma_2}$}, while the function $p_4(x)\pm q_1(x)y$ has a zero at
{\color{black}
$P_{\alpha}$}.
We note that $\gamma_1$ and $\gamma_2$ are zeros of $y$, thus we will have that
$p_4(x)=(\gamma_1-x)(\gamma_2-x)p_2(x)$, for a quadratic polynomial $p_2$.
Set $p_2(x)=ax^2+bx+c$ and $q_1(x)=dx+e$ and write the first $5$ terms of the Taylor expansion of $(\gamma_1-x)(\gamma_2-x)p_2(x)+q_1(x)y$:
\begin{gather*}
(A_0 e + c \gamma_1 \gamma_2) + (A_0 d + A_1 e - c \gamma_1 - c \gamma_2 + b \gamma_1 \gamma_2) x + (c + A_1 d +
A_2 e - b \gamma_1 - b \gamma_2 + a \gamma_1 \gamma_2) x^2
\\
+ (b + A_2 d + A_3 e - a \gamma_1 -
a \gamma_2) x^3 + (a + A_3 d + A_4 e) x^4.
\end{gather*}
All the coefficients of these terms should be equal to zero for some non-trivial values of $a$, $b$, $c$, $d$, $e$, which is equivalent to:
\begin{equation}\label{eq:det5}
\det\left(
\begin{matrix}
0 & 0 & \gamma_1\gamma_2 & 0 & A_0\\
0 & \gamma_1\gamma_2 & -\gamma_1-\gamma_2 & A_0 & A_1\\
\gamma_1\gamma_2 & -\gamma_1-\gamma_2 & 1 & A_1 & A_2\\
-\gamma_1-\gamma_2 & 1 & 0& A_2 & A_3\\
1 & 0 & 0 & A_3 & A_4
\end{matrix}
\right)
=0.
\end{equation}
Assuming that this condition is satisfied, we have that the following polynomial of degree $6$ has a zero of order $5$ at $x=0$:
\begin{equation}\label{eq:rho6}
\begin{aligned}
\rho_6(x)&=\frac{(p_4(x)+q_1(x)y)(p_4(x)-q_1(x)y)}{(\gamma_1-x)(\gamma_2-x)}
\\&=
(\gamma_1-x)(\gamma_2-x)p_2^2(x)-(a_1-x)(a_2-x)(a_3-x)q_1^2(x),
\end{aligned}
\end{equation}
and we need it to have also a zero at $x=\alpha$, for some $\alpha<a_3$.

Notice that $\rho_6$ is a polynomial with the positive leading coefficient, thus $\rho_6(-\infty)>0$. Also, $\rho_6(a_3)>0$. On the other hand $x=0$ is a zero of order $5$ of $\rho_6$, which means that the polynomial changes the sign at that point.
That implies that $\rho_6$ will always have another zero $\alpha$, satisfying $\alpha<a_3$.

\emph{Case 2: $\Q_{\alpha}$ is a hyperboloid.}
According to Lemma \ref{lemma:geom-type}, one of the caustics is a hyperboloid of the same geometric type as $\Q_{\alpha}$.
Since the number of bounces off the boundary $\E$ is odd, the other caustic is an ellipsoid.
Thus, we assume $\gamma_1\in(0,a_3)$, and $\gamma_2$, $\alpha$ are both in $(a_3,a_2)$ or both in $(a_2,a_1)$.
The condition for weak $5$-periodicity then is:
$$
5(P_0-P_{\gamma_1})\pm(P_{\alpha}-P_{\gamma_2})\sim0,
$$
where the positive sign corresponds to the reflection off $\Q_{\alpha}$ from inside, and the negative one to the reflection from outside.
That relation is equivalent to
$5P_0+P_{\alpha}^{\pm}+P_{\gamma_1}+P_{\gamma_2}\sim8P_{\infty}$.
Similarly as above, we get that the condition for weak $5$-periodicity is equivalent to the relations \refeq{eq:det5} and the existence of a quadratic polynomial $p_2$ and a linear one $q_1$ such that \refeq{eq:rho6} has a zero within the corresponding interval.
For $\gamma_2\in(a_3,a_2)$, we have that $\alpha\in(a_3,\gamma_2)$, and we check: $\rho_6(\gamma_2)<0$, $\rho(a_3)>0$, thus $\rho_6$ certainly has a zero in the requested interval.
For $\gamma_2\in(a_2,a_1)$, we have that $\alpha\in(\gamma_2,a_1)$, and we check:
$\rho_6(\gamma_2)<0$, $\rho_6(a_1)>0$, thus again $\rho_6$ certainly has a zero in the requested interval.

We conclude that the condition for weak $5$-periodicity is equivalent to the following:
\begin{itemize}
	\item $A_4=0$; or
	\item \refeq{eq:det5} and either one of the caustics is ellipsoid or both caustics are $1$-sheeted hyperboloids.
\end{itemize}

\paragraph*{$0$-weak period $6$.}
Suppose that the first and seventh segments of a trajectory are reflected to each other off quadric $\Q_{\alpha}$.

According to Lemma \ref{lemma:geom-type}, one of the caustics, say $\Q_{\gamma_1}$, is of the same geometric type as $\Q_{\alpha}$.
The following relation must be satisfied:
$$
6(P_0-P_{b_1})\pm(P_{\alpha}-P_{\gamma_1})\sim0,
$$
where the positive sign corresponds to the reflection off $\Q_{\alpha}$ from inside, and the negative one to the reflection from outside.
This is equivalent to $6P_0+P_{\alpha}^{\pm}+P_{\gamma_1}\sim 8P_{\infty}$.
Since the space $\LL(8 P_{\infty})$ is generated by $\{1,x,x^2, x^3, x^4, y, xy\}$, that will be equivalent to the existence of a quartic polynomial $p_4(x)$ and a linear polynomial $q_1(x)$ such that the function $p_4(x)+q_1(x)y$ has a zero of order $6$ at
{\color{black}
$P_0$} and a zero at
{\color{black}
$P_{\gamma_1}$}, while the function $p_4(x)\pm q_1(x)y$ has a zero at
{\color{black}
$P_{\alpha}$}.
We note that $\gamma_1$ is a zero of $y$, thus we will have that
$p_4(x)=(\gamma_1-x)p_3(x)$, for a cubic polynomial $p_3$.
Set $p_3(x)=ax^3+bx^2+cx+d$ and $q_1(x)=ex+f$ and write the first $6$ terms of the Taylor expansion of $(\gamma_1-x)p_3(x)+q_1(x)y$:
\begin{gather*}
A_0 f + d \gamma_1 + (-d + A_0 e + A_1 f + c \gamma_1) x + (-c + A_1 e + A_2 f +
b \gamma_1) x^2 +
\\
+ (-b + A_2 e + A_3 f + a \gamma_1) x^3 + (-a + A_3 e +
A_4 f) x^4 + (A_4 e + A_5 f) x^5.
\end{gather*}
All the coefficients of these terms should be equal to zero for some non-trivial values of $a$, $b$, $c$, $d$, $e$, $f$ which is equivalent to:
\begin{equation}\label{eq:det6}
\det\left(
\begin{matrix}
0 & 0 & 0 & -\gamma_1 & 0 & A_0\\
0 & 0 & -\gamma_1 & 1 & A_0 & A_1\\
0& -\gamma_1 & 1 & 0 & A_1 & A_2\\
-\gamma_1 & 1&  0 & 0 & A_2 & A_3\\
1 & 0 & 0 & 0 & A_3 & A_4\\
0 & 0 & 0 & 0 & A_4 & A_5
\end{matrix}
\right)
=0.
\end{equation}
Assuming that this condition is satisfied, we have that the following polynomial of degree $7$ has a zero of order $6$ at $x=0$:
\begin{equation}\label{eq:rho7}
\begin{aligned}
\rho_7(x)&=\frac{(p_4(x)+q_1(x)y)(p_4(x)-q_1(x)y)}{\gamma_1-x}
\\&
=
(\gamma_1-x)p_3^2(x)-q_1^2(x)(a_1-x)(a_2-x)(a_3-x)(\gamma_2-x),
\end{aligned}
\end{equation}
and we need that it also has a zero at $x=\alpha$, $\alpha$ lying in the same {\color{black}open} interval $(-\infty,a_3)$, $(a_3,a_2)$, $(a_2,a_1)$ as $\gamma_1$.

\emph{Case 1: $\Q_{\alpha}$ and $\Q_{\gamma_1}$ are ellipsoids.}
Notice that $\rho_7$ is an odd degree polynomial with the negative leading coefficient, thus $\rho_7(-\infty)>0$, and also, $\rho_7(\gamma_1)<0$.
From there $\rho_7$ must have a zero in $(-\infty, \gamma_1)$.

\emph{Case 2: $\Q_{\alpha}$ and $\Q_{\gamma_1}$ are $1$-sheeted hyperboloids.}
If $\Q_{\gamma_2}$ is an ellipsoid, then we need to show that $\rho_7$ has a zero in $(a_3,\gamma_1)$. This will be true since $\rho_7(\gamma_1)<0$ and $\rho_7(a_3)>0$.

If $\Q_{\gamma_2}$ is a $1$-sheeted hyperboloid, then we need to show that $\rho_7$ has a zero between $\gamma_1$ and $\gamma_2$.
That will be true since $\sign\rho_7(\gamma_1)=\sign(\gamma_2-\gamma_1)=-\sign\rho_7(\gamma_2)$.

If $\Q_{\gamma_2}$ is a $2$-sheeted hyperboloid, then we need to show that $\rho_7$ has a zero in $(\gamma_1,a_2)$.
That will be true since $\rho_7(\gamma_1)>0$ and $\rho_7(a_2)<0$.

\emph{Case 3: $\Q_{\alpha}$ and $\Q_{\gamma_1}$ are $2$-sheeted hyperboloids.}
Polynomial $\rho_7$ has a zero in $(\gamma_1,a_1)$ since
$\rho_7(\gamma_1)>0$ and $\rho_7(a_1)<0$.

We conclude that  the condition for weak $6$-periodicity is equivalent to \refeq{eq:det6}.

\paragraph*{$0$-weak  period $7$.}
Suppose that the first and eighth segments of a trajectory are reflected to each other off quadric $\Q_{\alpha}$.

\emph{Case 1: $\Q_{\alpha}$ is an ellipsoid.} Then one of the following relations must be satisfied:
\begin{itemize}
	\item $7(P_0-P_{b_1})\pm(P_{\alpha}-P_{b_1})\sim0$; or
	\item $7(P_0-P_{a_3})\pm(P_{\alpha}-P_{a_3})+P_{\gamma_1}-P_{\gamma_2}\sim0$, with $\Q_{\gamma_1}$ and $\Q_{\gamma_2}$ both being $1$-sheeted hyperboloids,
\end{itemize}
where the positive sign corresponds to the reflection off $\Q_{\alpha}$ from inside, and the negative one to the reflection from outside.

The first relation is equivalent to $7P_0+P_{\alpha}^{\pm}\sim 8P_{\infty}$.
Since the space $\LL(8 P_{\infty})$ is generated by
$\{1,x,x^2, x^3, x^4, y, xy\}$,
the relation will be equivalent to the existence of a quartic polynomial $p_4(x)$ and a linear polynomial $q_1(x)$ such that $p_4(x)+yq_1(x)$ has a zero of order $7$ at
{\color{black}
$P_0$} and $p_4(x)\pm yq_1(x)$ a zero at
{\color{black}
$P_{\alpha}$}.
The expression $p_4(x)+ yq_1(x)$ will have a zero of order $7$ at
{\color{black}
$P_0$} if and only if
$$
\det\left(
\begin{array}{cc}
A_4 & A_5\\
A_5 & A_6
\end{array}
\right)=0,
$$
where coefficients $A_j$ are defined as in Lemma \ref{lemma:cayley}.
Assuming that this condition is satisfied, we have that the following polynomial of degree $8$ has a zero of order $7$ at $x=0$:
\begin{equation*}
\begin{aligned}
\rho_8(x)\ &=(p_4(x)+yq_1(x))(p_3(x)-yq_1(x))\\
&=
p_4^2(x)-(a_1-x)(a_2-x)(a_3-x)(\gamma_1-x)(\gamma_2-x)q_1^2(x),
\end{aligned}
\end{equation*}
and we need that it also has a zero at $x=\alpha$, $\alpha$ lying in the {\color{black}open} interval $(-\infty,a_3)$.
Notice that $\rho_8$ is a polynomial of degree $8$ with the positive leading coefficient, thus $\rho_8(-\infty)>0$. Also, $\rho_8(a_3)>0$. On the other hand $x=0$ is its zero of order $7$, which means that the polynomial changes the sign at that point.
That means that $\rho$ will always have another zero $\alpha$, satisfying $\alpha<a_3$.

The second relation is equivalent to
$7P_0+P_{\alpha}^{\pm}+P_{\gamma_1}+P_{\gamma_2}\sim 10P_{\infty}$.
Since the space $\LL(10 P_{\infty})$ is generated by
$\{1,x,x^2, x^3, x^4, x^5, y, xy, x^2y\}$,
that will be equivalent to the existence of polynomials $p_5(x)$ and $q_2(x)$ of degrees $5$ and $2$ respectfully, such that the function $p_5(x)+q_2(x)y$ has a  zero of order $7$ at
{\color{black}
$P_0$}
 and zeros at
{\color{black}
 $P_{\gamma_1}$, $P_{\gamma_2}$}, while the function $p_5(x)\pm q_2(x)y$ has a zero at
{\color{black}
$P_{\alpha}$}.
We note that $\gamma_1$ and $\gamma_2$ are zeros of $y$, thus we will have that
$p_5(x)=(\gamma_1-x)(\gamma_2-x)p_3(x)$, for a cubic polynomial $p_3$.
Set $p_3(x)=ax^3+bx^2+cx+d$ and $q_2(x)=ex^2+fx+g$ and write the first $7$ terms of the Taylor expansion of $(\gamma_1-x)(\gamma_2-x)p_3(x)+q_2(x)y$:
\begin{gather*}
(A_0 g + d \gamma_1 \gamma_2) + (A_0 f + A_1 g - d \gamma_1 - d \gamma_2 + c \gamma_1 \gamma_2) x
\\
+ (d + A_0 e +
A_1 f + A_2 g - c \gamma_1 - c \gamma_2 + b \gamma_1 \gamma_2) x^2
\\
+ (c + A_1 e + A_2 f +
A_3 g - b \gamma_1 - b \gamma_2 + a \gamma_1 \gamma_2) x^3 + (b + A_2 e + A_3 f + A_4 g -
a \gamma_1 - a \gamma_2) x^4
\\
+ (a + A_3 e + A_4 f + A_5 g) x^5 + (A_4 e + A_5 f +
A_6 g) x^6
\end{gather*}
All the coefficients of these terms should be equal to zero for some non-trivial values of $a$, $b$, $c$, $d$, $e$, $f$, $g$ which is equivalent to:
\begin{equation}\label{eq:det7}
\det\left(
\begin{matrix}
0 & 0 & 0 & \gamma_1\gamma_2 & 0 & 0 & A_0\\
0 & 0 &  \gamma_1\gamma_2 &  -\gamma_1-\gamma_2 & 0 & A_0 & A_1\\
0& \gamma_1\gamma_2 & -\gamma_1-\gamma_2 & 1 & A_0 & A_1 & A_2\\
\gamma_1\gamma_2 & -\gamma_1-\gamma_2 & 1 & 0& A_1&  A_2 & A_3\\
-\gamma_1-\gamma_2 & 1 & 0 & 0 & A_2 & A_3 & A_4\\
1 & 0 & 0 & 0 & A_3 & A_4 & A_5\\
0 & 0 & 0 & 0 & A_4 & A_5 & A_6
\end{matrix}
\right)
=0.
\end{equation}
Assuming that this condition is satisfied, we have that the following polynomial of degree $8$ has a zero of order $7$ at $x=0$:
\begin{equation}\label{eq:rho8}
\begin{aligned}
\rho_8(x)&=\frac{(p_5(x)+q_2(x)y)(p_5(x)-q_2(x)y)}{(\gamma_1-x)(\gamma_2-x)}
\\&=
(\gamma_1-x)(\gamma_2-x)p_3^2(x)-(a_1-x)(a_2-x)(a_3-x)q_2^2(x),
\end{aligned}
\end{equation}
and we need it to have also a zero at $x=\alpha$, for some $\alpha<a_3$.

Notice that $\rho_8$ is a polynomial with the positive leading coefficient, thus $\rho_8(-\infty)>0$. Also, $\rho_8(a_3)>0$. On the other hand $x=0$ is a zero of order $7$ of $\rho_8$, which means that the polynomial changes the sign at that point.
That implies that $\rho_8$ will always have another zero $\alpha$, satisfying $\alpha<a_3$.

\emph{Case 2: $\Q_{\alpha}$ is a hyperboloid.}
According to Lemma \ref{lemma:geom-type}, one of the caustics is a hyperboloid of the same geometric type as $\Q_{\alpha}$.
Since the number of bounces off the boundary $\E$ is odd, the other caustic is an ellipsoid.
Thus, we assume $\gamma_1\in(0,a_3)$, and $\gamma_2$, $\alpha$ are both in $(a_3,a_2)$ or both in $(a_2,a_1)$.
The condition for weak $7$-periodicity then is:
$$
7(P_0-P_{\gamma_1})\pm(P_{\alpha}-P_{\gamma_2})\sim0,
$$
where the positive sign corresponds to the reflection off $\Q_{\alpha}$ from inside, and the negative one to the reflection from outside.
That relation is equivalent to
$7P_0+P_{\alpha}^{\pm}+P_{\gamma_1}+P_{\gamma_2}\sim10P_{\infty}$.
Similarly as above, we get that the condition for weak $7$-periodicity is equivalent to the relations \refeq{eq:det7} and the existence of a cubic polynomial $p_3$ and a quadratic one $q_2$ such that \refeq{eq:rho8} has a zero within the corresponding interval.
For $\gamma_2\in(a_3,a_2)$, we have that $\alpha\in(a_3,\gamma_2)$, and we check: $\rho_8(\gamma_2)<0$, $\rho_8(a_3)>0$, thus $\rho_8$ certainly has a zero in the requested interval.
For $\gamma_2\in(a_2,a_1)$, we have that $\alpha\in(\gamma_2,a_1)$, and we check:
$\rho_8(\gamma_2)<0$, $\rho_8(a_1)>0$, thus again $\rho_8$ certainly has a zero in the requested interval.

We conclude that the condition for weak $7$-periodicity is equivalent to the following:
\begin{itemize}
	\item $\det\left(
	\begin{array}{cc}
	A_4 & A_5\\
	A_5 & A_6
	\end{array}
	\right)=0$; or
	\item \refeq{eq:det7} and either one of the caustics is ellipsoid or both caustics are $1$-sheeted hyperboloids.
\end{itemize}

\subsection{$1$-weak periodic trajectories in dimension $4$}\label{sec:1weakdim4}
Let us study trajectories  within an ellipsoid in the $4$-dimensional space which are $1$ weak $n$-periodic.
More precisely, assume a trajectory becomes periodic in elliptic coordinates after $n$ reflections off ellipsoid $\E$, one reflection of quadric $\Q_{\alpha_1}$, and one off $\Q_{\alpha_2}$.
Additionally, we will suppose that $\Q_{\alpha_1}$ and $\Q_{\alpha_2}$ are of the same type.

The divisor condition for such a trajectory is:
$$
n(P_0-P_{b_1})\pm(P_{\alpha_1}-P_{b_{j_1}})\pm(P_{\alpha_2}-P_{b_{j_2}})
+
\varepsilon_1(P_{b_2}-P_{b_3})
+
\varepsilon_2(P_{b_4}-P_{b_5})
+
\varepsilon_3(P_{b_6}-P_{b_7})
\sim
0,
$$
with $\varepsilon_k\in\{0,1\}$ and polynomial $\Pol$ is positive everywhere between $\alpha_k$ and $b_{j_k}$.
This is equivalent to:
$$
nP_0+\eta P_{b_1}+P_{\alpha_1}^{\pm}+P_{\alpha_2}^{\pm}
+
\varepsilon_1'(P_{b_2}+P_{b_3})
+
\varepsilon_2'(P_{b_4}+P_{b_5})
+
\varepsilon_3'(P_{b_6}+P_{b_7})
\sim
2mP_{\infty},
$$
with $2m=n+\eta+2+2\varepsilon_1'+2\varepsilon_2'+2\varepsilon_3'$, and $\eta=1$ if $n$ is odd and $\eta=0$ if $n$ is even.
The basis of $\LL(2mP_{\infty})$ is:
$$
1,x,x^2,\dots,x^m, y, xy, \dots, x^{m-4}y,
$$
so the divisor condition is equivalent to the existence of polynomials $p_m(x)$ and $q_{m-4}(x)$ such that $p_m(x)+yq_{m-4}(x)$ has a zero of order $n$ at
{\color{black}
$P_0$}, zero of order $\eta$ at
{\color{black}
$P_{b_1}$}, zeros of order $\varepsilon_k'$ at {\color{black}$P_{b_{2k}}$} and
{\color{black}$P_{b_{2k+1}}$},
while $p_m(x)\pm yq_{m-4}(x)$ has zeros at
{\color{black}
$P_{\alpha_1}$ and $P_{\alpha_2}$}.
Then the polynomial
$$
\begin{aligned}
\rho_{n+2}(x)
\ &=
\frac{(p_m(x)+yq_{m-4}(x))(p_m(x)-yq_{m-4}(x))}
{\rho_1(x)}
\\
&=
\rho_1(x)p_{\frac12(n-\eta)+1}^2(x)-\rho_2(x)q_{m-4}^2(x),
\end{aligned}
$$
has a zero of order $n$ at $x=0$ and simple zeros $x=\alpha_1$ and $x=\alpha_2$,
where $\rho_1\rho_2=\Pol$ and
$$
\rho_1(x)=(b_1-x)^{\eta}(b_2-x)^{\varepsilon_1'} (b_3-x)^{\varepsilon_1'} (b_4-x)^{\varepsilon_2'}
(b_5-x)^{\varepsilon_2'} (b_6-x)^{\varepsilon_3'} (b_7-x)^{\varepsilon_3'}.
$$
That implies:
$$
\rho_1(x)p_{\frac12(n-\eta)+1}^2(x)-\rho_2(x)q_{m-4}^2(x)
=
x^n r_2(x).
$$
Now, relation (\ref{eq:weak-pell}) derived from there implies that $r_2$ is positive where $\Pol$ is negative.
Also, $r_2$ has the same signs at the endpoints of the {\color{black}closed intervals} $[b_{2j},b_{2j+1}]$, its zeros correspond to two quadrics of the same type.


\subsection*{Acknowledgment} {\color{black} The authors are grateful to Klaus Schiefermayr for inspirative discussions. The authors thank the referees for numerous suggestions and comments which significantly helped to improve the presentation.
The research  was supported
by the Australian Research Council, Discovery Project 190101838 \emph{Billiards within quadrics and beyond}, by the Simons Foundation grant no. 854861,  by Mathematical Institute of the Serbian Academy of Sciences and Arts, the Science Fund of Serbia grant Integrability and Extremal Problems in Mechanics, Geometry and
Combinatorics, MEGIC, Grant No. 7744592 and the Ministry for Education, Science, and Technological Development of Serbia.
}


\begin{bibdiv}
\begin{biblist}

\bib{AbendaFed2006}{article}{
    title={Closed geodesics and billiards on quadrics related to elliptic KdV solutions},
    author={Abenda, Simonetta},
    author={Fedorov, Yuri},
    journal={Letters in Mathematical Physics},
    volume={76},
    date={2006},
    pages={111--134}
}

\bib{AbendaGrin2010}{article}{
   author={Abenda, Simonetta},
   author={Grinevich, Petr G.},
   title={Periodic billiard orbits on $n$-dimensional ellipsoids with
   impacts on confocal quadrics and isoperiodic deformations},
   journal={J. Geom. Phys.},
   volume={60},
   date={2010},
   number={10},
   pages={1617--1633}
}


\bib{ADR2019rcd}{article}{
	AUTHOR = {Adabrah, Anani Komla},
AUTHOR = {Dragovi\'{c}, Vladimir},
AUTHOR = {Radnovi\'{c},
		Milena},
	TITLE = {Periodic billiards within conics in the {M}inkowski plane and
		{A}khiezer polynomials},
	JOURNAL = {Regul. Chaotic Dyn.},
	FJOURNAL = {Regular and Chaotic Dynamics. International Scientific
		Journal},
	VOLUME = {24},
	YEAR = {2019},
	NUMBER = {5},
	PAGES = {464--501},
	ISSN = {1560-3547},
	MRCLASS = {37D50 (26C05 70H06)},
	MRNUMBER = {4015392},
	DOI = {10.1134/S1560354719050034},
	URL = {https://doi.org/10.1134/S1560354719050034},
}	
	



\bib{AhiezerAPPROX}{book}{
   author={Ahiezer, N. I.},
   title={Lekcii po Teorii Approksimacii},
   language={Russian},
   publisher={OGIZ, Moscow-Leningrad},
   date={1947},
   pages={323}
}

\bib{ADR2020rj}{article}{
	AUTHOR = {Andrews, George E.},
AUTHOR = {Dragovi\'{c}, Vladimir},
AUTHOR = {Radnovi\'{c},
		Milena},
	TITLE = {Combinatorics of periodic ellipsoidal billiards},
	JOURNAL = {The Ramanujan Journal, arXiv: 1908.01026.},
	FJOURNAL = {},
	VOLUME = {},
	YEAR = {},
	NUMBER = {},
	PAGES = {},
	ISSN = {},
	MRCLASS = {},
	MRNUMBER = {},
	DOI = {},
	URL = {}
}	


\bib{ArnoldMMM}{book}{
   author={Arnol\cprime d, V. I.},
   title={Mathematical methods of classical mechanics},
   series={Graduate Texts in Mathematics},
   volume={60},
   edition={2},
   note={Translated from the Russian by K. Vogtmann and A. Weinstein},
   publisher={Springer-Verlag, New York},
   date={1989},
   pages={xvi+508}
}


\bib{Audin1994}{article}{
    author={Audin, Mich\`ele},
    title={Courbes alg\'ebriques et syst\`emes int\'egrables:
g\'eodesiques des quadriques},
    journal={Exposition. Math.},
    volume={12},
    date={1994},
    pages={193--226}
}

\bib{ADSK}{article}{
	author={Avila, Artur},
	author={De Simoi, Jacopo},
	author={Kaloshin, Vadim},
	title={An integrable deformation of an ellipse of small eccentricity is
		an ellipse},
	journal={Ann. of Math. (2)},
	volume={184},
	date={2016},
	number={2},
	pages={527--558},
}



\bib{bm}{article}{
	author={Bialy, Misha},
	author={Mironov, Andrey E.},
	title={Angular billiard and algebraic Birkhoff conjecture},
	journal={Adv. Math.},
	volume={313},
	date={2017},
	pages={102--126},
}

\bib{bm1}{article}{
	author={Bialy, Misha},
	author={Mironov, Andrey E.},
	title={The Birkhoff-Poritsky conjecture for centrally-symmetric billiard
		tables},
	journal={Ann. of Math. (2)},
	volume={196},
	date={2022},
	number={1},
	pages={389--413},
}

\bib{bol}{article}{
	author={Bolotin, S. V.},
	title={Integrable billiards on surfaces of constant curvature},
	language={Russian},
	journal={Mat. Zametki},
	volume={51},
	date={1992},
	number={2},
	pages={20--28, 156},
	translation={
		journal={Math. Notes},
		volume={51},
		date={1992},
		number={1-2},
		pages={117--123},
	},
}



\bib{BDFRR2002}{article}{
   author={Bolotin, S.},
   author={Delshams, A.},
   author={Fedorov, Yu.},
   author={Ram\'\i rez-Ros, R.},
   title={Bi-asymptotic billiard orbits inside perturbed ellipsoids},
   conference={
      title={Progress in nonlinear science, Vol. 1},
      address={Nizhny Novgorod},
      date={2001},
   },
   book={
      publisher={RAS, Inst. Appl. Phys., Nizhniy Novgorod},
   },
   date={2002},
   pages={48--62}
}



\bib{CRR2011}{article}{
   author={Casas, Pablo S.},
   author={Ram\'\i rez-Ros, Rafael},
   title={The frequency map for billiards inside ellipsoids},
   journal={SIAM J. Appl. Dyn. Syst.},
   volume={10},
   date={2011},
   number={1},
   pages={278--324}
}

\bib{CRR2012}{article}{
   author={Casas, Pablo S.},
   author={Ram\'\i rez-Ros, Rafael},
   title={Classification of symmetric periodic trajectories in ellipsoidal
   billiards},
   journal={Chaos},
   volume={22},
   date={2012},
   number={2},
   pages={026110, 24}
}



\bib{Cayley1854}{article}{
    author={Cayley, Arthur},
    title={Developments on the porism of the in-and-circumscribed polygon},
    journal={Philosophical magazine},
    volume={7},
    date={1854},
    pages={339--345}
}


\bib{CCS1993}{article}{
    author={Chang, Shau-Jin},
    author={Crespi, Bruno},
    author={Shi, Kang-Jie},
    title={Elliptical billiard systems and the full Poncelet's theorem in $n$ dimensions},
    journal={J. Math. Phys.},
    volume={34},
    number={6},
    date={1993},
    pages={2242--2256}
}




\bib{Dar1870}{article}{
    author={Darboux, Gaston},
    title={
    Sur les polygones inscrits et circonscrits a l'ellipsoide
        },
    journal={Bulletin de
la Societe philomathique},
      volume={7},
      date={1870},
      pages={92-94}
}


\bib{DarbouxSUR}{book}{
    author={Darboux, Gaston},
    title={
        Le\c{c}ons sur la th\'eorie
        g\'en\'erale des surfaces et les
        applications g\'eo\-m\'etri\-ques du
        calcul infinitesimal
    },
    publisher={Gauthier-Villars},
    address={Paris},
     date={1914},
     volume={2 and 3}
}







\bib{DragRadn1998a}{article}{
    author={Dragovi\'c, Vladimir},
    author={Radnovi\'c, Milena},
    title={Conditions of Cayley's type for ellipsoidal billiard},
    journal={J. Math. Phys.},
    volume={39},
    date={1998},
    number={1},
    pages={355--362}
}

\bib{DragRadn1998b}{article}{
    author={Dragovi\'c, Vladimir},
    author={Radnovi\'c, Milena},
    title={Conditions of Cayley's type for ellipsoidal billiard},
    journal={J. Math. Phys.},
    volume={39},
    date={1998},
    number={11},
    pages={5866--5869}
}


\bib{DragRadn2004}{article}{
    author={Dragovi\'c, Vladimir},
    author={Radnovi\'c, Milena},
    title={Cayley-type conditions for billiards within $k$ quadrics in $\mathbf R^d$},
    journal={J. of Phys. A: Math. Gen.},
    volume={37},
    pages={1269--1276},
    date={2004}
}






\bib{DragRadn2006jmpa}{article}{
    author={Dragovi\'c, Vladimir},
    author={Radnovi\'c, Milena},
    title={Geometry of integrable billiards and pencils of quadrics},
    journal={Journal Math. Pures Appl.},
    volume={85},
    date={2006},
    pages={758--790}
}


\bib{DragRadn2008}{article}{
    author={Dragovi\'c, Vladimir},
    author={Radnovi\'c, Milena},
    title={Hyperelliptic Jacobians as Billiard Algebra of Pencils of Quadrics: Beyond Poncelet Porisms},
    journal={Adv.  Math.},
    volume={219},
    date={2008},
    number={5},
    pages={1577--1607}
}



\bib{DragRadn2011book}{book}{
    author={Dragovi\'c, Vladimir},
    author={Radnovi\'c, Milena},
    title={Poncelet Porisms and Beyond},
    publisher={Springer Birkhauser},
    date={2011},
    place={Basel}
}

\bib{DragRadn2012adv}{article}{
	author={Dragovi\'c, Vladimir},
	author={Radnovi\'c, Milena},
	title={Ellipsoidal billiards in pseudo-Euclidean spaces and relativistic quadrics},
	journal={Advances in Mathematics},
	volume={231},
	pages={1173--1201},
	date={2012}
}

\bib{DR1}{article}{
	author={Dragovi\'{c}, Vladimir},
	author={Radnovi\'{c}, Milena},
	title={Pseudo-integrable billiards and arithmetic dynamics},
	journal={J. Mod. Dyn.},
	volume={8},
	date={2014},
	number={1},
	pages={109--132},
	issn={1930-5311},
	review={\MR{3296569}},
	doi={10.3934/jmd.2014.8.109},
}


\bib{DragRadn2018}{article}{
	AUTHOR = {Dragovi\'{c}, Vladimir},
AUTHOR = {Radnovi\'{c}, Milena},
	TITLE = {Periodic {E}llipsoidal {B}illiard {T}rajectories and
		{E}xtremal {P}olynomials},
	JOURNAL = {Comm. Math. Phys.},
	FJOURNAL = {Communications in Mathematical Physics},
	VOLUME = {372},
	YEAR = {2019},
	NUMBER = {1},
	PAGES = {183--211},
	ISSN = {0010-3616},
	MRCLASS = {37D50 (33C47)},
	MRNUMBER = {4031799},
	DOI = {10.1007/s00220-019-03552-y},
	URL = {https://doi.org/10.1007/s00220-019-03552-y},
}


\bib{DragRadn2019rcd}{article}{
	author={Dragovi\'{c}, Vladimir},
	author={Radnovi\'{c}, Milena},
	title={Caustics of Poncelet polygons and classical extremal polynomials},
	journal={Regul. Chaotic Dyn.},
	volume={24},
	date={2019},
	number={1},
	pages={1--35}
}


\bib{DuistermaatBOOK}{book}{
   author={Duistermaat, Johannes J.},
   title={Discrete integrable systems: QRT maps and elliptic surfaces},
   series={Springer Monographs in Mathematics},
   publisher={Springer},
   place={New York},
   date={2010},
   pages={xxii+627},
   isbn={978-1-4419-7116-6}
}


\bib{Fed2001}{article}{
    author={Fedorov, Yuri},
    title={An ellipsoidal billiard with quadratic potential},
    journal={Funct. Anal. Appl.},
    volume={35},
    date={2001},
    number={3},
    pages={199--208}
}

\bib{glu1}{article}{
	author={Glutsyuk, Alexey},
	title={On polynomially integrable Birkhoff billiards on surfaces of
		constant curvature},
	journal={J. Eur. Math. Soc. (JEMS)},
	volume={23},
	date={2021},
	number={3},
	pages={995--1049},
}



\bib{GrifHar1978}{article}{
    author={Griffiths, Philip},
    author={Harris, Joe},
    title={On Cayley's explicit solution to Poncelet's porism},
    journal={EnsFeign. Math.},
    volume={24},
    date={1978},
    number={1-2},
    pages={31--40}
}


\bib{Hal1888}{book}{
    author={Halphen, G.-H.},
    title={Trait\' e des fonctiones elliptiques et de leures applications},
    part={deuxieme partie},
    publisher={Gauthier-Villars et fils},
    address={Paris},
    date={1888}
}


\bib{Jac}{book}{
author={Jacobi, C.},
title={Jacobi's Lectures on Dynamics},
publisher={Springer, TRIM 51},
date={2009}
}

\bib{KS}{article}{
	author={Kaloshin, Vadim},
	author={Sorrentino, Alfonso},
	title={On the local Birkhoff conjecture for convex billiards},
	journal={Ann. of Math. (2)},
	volume={188},
	date={2018},
	number={1},
	pages={315--380},
	issn={0003-486X},
	review={\MR{3815464}},
	doi={10.4007/annals.2018.188.1.6},
}


\bib{KTBilliards}{book}{
	author={Kozlov, Valeri\u{\i} V.},
	author={Treshch\"{e}v, Dmitri\u{\i} V.},
	title={Billiards},
	series={Translations of Mathematical Monographs},
	volume={89},
	note={A genetic introduction to the dynamics of systems with impacts;
		Translated from the Russian by J. R. Schulenberger},
	publisher={American Mathematical Society, Providence, RI},
	date={1991},
}

\bib{KLN1990}{article}{
   author={Kre\u\i n, M. G.},
   author={Levin, B. Ya.},
   author={Nudel\cprime man, A. A.},
   title={On special representations of polynomials that are positive on a
   system of closed intervals, and some applications},
   note={Translated from the Russian by Lev J. Leifman and Tatyana L.
   Leifman},
   conference={
      title={Functional analysis, optimization, and mathematical economics},
   },
   book={
      publisher={Oxford Univ. Press, New York},
   },
   date={1990},
   pages={56--114}
}




\bib{MV1991}{article}{
   author={Moser, J.},
   author={Veselov, A. },
   title={Discrete versions of some classical integrable
systems and factorization of matrix polynomials},
   journal={Comm. Math. Phys.},
   volume={139},
   date={1991},
   number={2},
   pages={217--243}
}



\bib{PS1999}{article}{
   author={Peherstorfer, F.},
   author={Schiefermayr, K.},
   title={Description of extremal polynomials on several intervals and their
   computation. I, II},
   journal={Acta Math. Hungar.},
   volume={83},
   date={1999},
   number={1-2},
   pages={27--58, 59--83}
}




\bib{Poncelet1822}{book}{
    author={Poncelet, Jean Victor},
    title={Trait\'e des propri\'et\'es projectives des figures},
    publisher={Mett},
    address={Paris},
    date={1822}
}




\bib{Radn2015}{article}{
    author={Radnovi\'c, Milena},
    title={Topology of the elliptical billiard with the Hooke's potential},
    journal={Theoretical and Applied Mechanics},
    date={2015},
    volume={42},
    number={1},
    pages={1--9}
}



\bib{RR2014}{article}{
   author={Ram\'\i rez-Ros, Rafael},
   title={On Cayley conditions for billiards inside ellipsoids},
   journal={Nonlinearity},
   volume={27},
   date={2014},
   number={5},
   pages={1003--1028}
}

\bib{Si2011}{book}{
    author={Simon, Barry},
    title={
    Szeg\"o's Theorem and its Descendants
    },
    publisher={Princeton University Press},
    address={Princetin and Oxford},
     date={2011}
}

\bib{Si2015}{book}{
    author={Simon, Barry},
    title={Operator Theory, A Comperhansive Course in Analysis. Part 4 },
    publisher={AMS},
     date={2015}
}


\bib{SodinYu1992}{article}{
   author={Sodin, M. L.},
   author={Yuditski\u\i , P. M.},
   title={Functions that deviate least from zero on closed subsets of the
   real axis},
   language={Russian, with Russian summary},
   journal={Algebra i Analiz},
   volume={4},
   date={1992},
   number={2},
   pages={1--61},
   issn={0234-0852},
   translation={
      journal={St. Petersburg Math. J.},
      volume={4},
      date={1993},
      number={2},
      pages={201--249},
      issn={1061-0022},
   },
}

\bib{Sp1957}{book}{
   author={Springer, George},
   title={Introduction to Riemann Surfaces},
   publisher={AMS Chelsea Publishing},
   date={1957},
   pages={309}
}

\bib{Tab}{book}{
	author={Tabachnikov, Serge},
	title={Geometry and billiards},
	series={Student Mathematical Library},
	volume={30},
	publisher={American Mathematical Society, Providence, RI; Mathematics
		Advanced Study Semesters, University Park, PA},
	date={2005},
	pages={xii+176},
	isbn={0-8218-3919-5},
	review={\MR{2168892}},
	doi={10.1090/stml/030},
}

\bib{Tjurin1975}{article}{
	author={Tjurin, A. N.},
	title={The intersection of quadrics},
	language={Russian},
	journal={Uspehi Mat. Nauk},
	volume={30},
	date={1975},
	number={6(186)},
	pages={51--99},
}



\bib{Wiersig2000}{article}{
   author={Wiersig, Jan},
   title={Ellipsoidal billiards with isotropic harmonic potentials},
   journal={Internat. J. Bifur. Chaos Appl. Sci. Engrg.},
   volume={10},
   date={2000},
   number={9},
   pages={2075--2098}
}


\bib{WDullin2002}{article}{
    author={Waalkens, H.},
    author={Dullin, H. R.},
    title={Quantum Monodromy in Prolate Ellipsoidal Billiards},
    journal={Annals of Physics},
    volume={295},
    date={2002},
    number={1},
    pages={81--112}
}







\end{biblist}
\end{bibdiv}
\end{document}